\documentclass{article}
\usepackage{latexsym, amsmath, amssymb}
\usepackage{graphicx}
\usepackage{epstopdf}
\usepackage{amsmath, amsthm, amscd, amsfonts}
\usepackage{amsmath}
\usepackage{array}
\usepackage{multirow}
\usepackage{amssymb}
\usepackage{float}
\usepackage{enumerate}
\usepackage[a4paper, total={6in, 10in}]{geometry}

\newtheorem{thm}{Theorem}[section]

\newtheorem{lem}{Lemma}
\theoremstyle{definition} 
\theoremstyle{question} 
\theoremstyle{remark} 

\makeatletter
\newcommand*{\rom}[1]{\expandafter\@slowromancap\romannumeral #1@}
\makeatother

\def\bege{\begin{equation}} \def\ende{\end{equation}}

   \def\begr{\begin{eqnarray}}
\def\endr{\end{eqnarray}} 
 
\def\bege{\begin{equation}} \def\ende{\end{equation}}
\def\begr{\begin{eqnarray}} \def\endr{\end{eqnarray}} \def\bnum{\begin{enumerate}} \def\enum{\end{enumerate}}
\begin{document}

\begin{center}
\textbf{UNBOUNDED MIXED RESOLVABILITY OF WEB GRAPH AND PRISM RELATED GRAPH}
\end{center}

\begin{center}
Sunny Kumar Sharma and Vijay Kumar Bhat
\end{center}

\begin{center}
School of Mathematics, Shri Mata Vaishno Devi University, Katra-$182320$, Jammu and Kashmir, India.
\end{center}
\begin{center}
  1. sunnysrrm94@gmail.com 2. vijaykumarbhat2000@yahoo.com
\end{center}
\textbf{Abstract} Let $\mathbb{E}(H)$ and $\mathbb{V}(H)$ denote the edge set and the vertex set of the simple connected graph $H$, respectively. The mixed metric dimension of the graph $H$ is the graph invariant, which is the mixture of two important graph parameters, the edge metric dimension and the metric dimension. In this article, we compute the mixed metric dimension for the two families of the plane graphs viz., the Web graph $\mathbb{W}_{n}$ and the Prism allied graph $\mathbb{D}_{n}^{t}$. We show that the mixed metric dimension is non-constant unbounded for these two families of the plane graph. Moreover, for the Web graph $\mathbb{W}_{n}$ and the Prism allied graph $\mathbb{D}_{n}^{t}$, we unveil that the mixed metric basis set $M_{G}^{m}$ is independent.

\section{Introduction and Preliminaries}
The graph invariant $metric$ $dimension$ is among the important and highly active research topic in Graph Theory. This fundamental concept of metric dimension was founded by two groups of researchers independently viz., Slater in \cite{ps} and Harary and Melter in \cite{fr}, in the late seventies. The set $M_{G}$ of points in the taken metric space with the property (or characteristic) that any point of the space is determined uniquely by its distances from the points of $M_{G}$, is referred to as the $generator$ (or $metric$ $generator$) of the given metric space. These metric generating sets are called the locating sets by Slater in \cite{ps} and the resolving sets by Melter and Harary in \cite{fr}, respectively.

After these two important initial papers \cite{ps,fr}, several works regarding theoretical properties, as well as applications, of this graph invariant were published. Initially, Slater considered special acknowledgment of a thief in the network, while others noticed problems in picture preparing (or image processing) and design acknowledgment (or pattern recognition) \cite{mt}, applications to science are given in \cite{ce}, to the route of exploring specialist (navigating agent or robots) in systems (or networks) are examined in \cite{srb}, to issues of check and system revelation (or network discovery) in \cite{zf}, application to combinatorial enhancement (or optimization) is yielded in \cite{co}, and for more work see \cite{sv,ssv,im}.\par

Suppose $\mathbb{E}(H)$ and $\mathbb{V}(H)$ denote the edge set and the vertex set of the simple connected graph $H$, respectively. The distance between two vertices $\alpha,\beta \in \mathbb{V}(H)$, is denoted and defined as $d_{H}(\alpha,\beta)=$ length of the shortest possible $\alpha-\beta$ path in $H$ and the $d_{H}(\alpha,\varepsilon)=min\{d_{H}(\alpha,\beta_{1}), d_{H}(\alpha,\\ \beta_{2})\}$ represents the distance between an edge $\varepsilon=\beta_{1}\beta_{2}$ and the vertex $\alpha$ in $H$. If the distance between the vertex $\alpha$ and an element $\beta$ is not equaled to the distance between the same vertex $\alpha$ and an element $\gamma$ in $H$ (where $\beta,\gamma \in \mathbb{V}(H)\cup\mathbb{E}(H)$), then one can say that the vertex $\alpha$ distinguish (determines or resolves) two elements $\beta$ and $\gamma$ in $H$.\par

A set $M_{G}$ consisting of the vertices of the graph $H$, is termed as the $metric$ $generator$ for $H$, if the vertices of $M_{G}$ distinguish (determines or resolves) every pair of different vertices of the connected graph $H$. These metric generators are called $metric$ $basis$ for $H$ if it has the minimum cardinality and this cardinal number of the metric basis is referred to as the $metric$ $dimension$ of the graph $H$, denoted by $\beta(H)$ or $dim(H)$.

On the other hand, concerning the hypothetical examinations of this important topic, various perspectives of metric generators $M_{G}$ have been depicted in the recent literature, which has profoundly added to acquire more understanding into numerical properties of this graph invariant related with distances in networks. Several authors working on this topic have presented different varieties of metric generators like for example, independent resolving sets, resolving dominating sets, strong resolving sets, local metric sets, edge resolving sets, strong resolving partitions, mixed metric sets, etc. For these see references in \cite{flc, flm, sv, ssv1}. A set $L$ consisting of vertices of the graph $H$ is said to be an independent resolving set for $H$, if $L$ is both resolving (metric generator) and independent.\par

One can see that the metric dimension deals with the vertices of the graph by its definition, a similar concept dealing with the edges of the graph introduced by Kelenc et al. in \cite{flc}, called the edge metric dimension of the graph $H$, which uniquely identifies the edges related to a graph $H$. For an edge $\varepsilon=\beta_{1}\beta_{2}$ and a vertex $x$ the distance between them is defined as $d_{H}(x,\varepsilon)=min\{d_{H}(x,\beta_{1}),d_{H}(x,\beta_{2})\}$. A subset $M_{G}^{\varepsilon}$ is called an edge metric generator for $H$, if any two different edges of $H$ are distinguish by some vertex of $M_{G}^{\varepsilon}$. The edge metric generator with minimum cardinality is termed as edge metric basis and that cardinality is known as the edge metric dimension of the graph $H$, and which is denoted by $edim(H)$ or $\beta_{E}(H)$. A set $L_{E}$ consisting of vertices of the graph $H$ is said to be an independent edge metric generator for $H$, if $L_{E}$ is both edge metric generator and independent.

\textbf{Mixed metric dimension:}
Recently, a new kind of graph parameter was introduced by Kelenc et al. in \cite{flm}, which is the composition of both, the edge metric dimension and the metric dimension and called the mixed metric dimension for a graph $H$. A subset $M_{G}^{m}$ is called a mixed metric generator for $H$, if any two different elements of $\mathbb{V}(H)\cup\mathbb{E}(H)$ are distinguished by some vertex of $M_{G}^{m}$. For an ordered subset $\mathfrak{L}_{M}=\{\zeta_{1}, \zeta_{2}, \zeta_{3},...,\zeta_{p}\}$ of vertices of the graph $H$, and an element $y\in \mathbb{V}(H)\cup \mathbb{E}(H)$, the mixed metric code$\backslash$mixed metric representation of $y$ regarding $M_{G}^{m}$ is the ordered $p$-tuple $\zeta_{M}(y|M_{G}^{m})=(d_{H}(y,\zeta_{1}), d_{H}(y,\zeta_{2}), d_{H}(y,\zeta_{3}),...,d_{H}(y,\zeta_{p}))$.

If for any two distinct elements $y_{1}$ and $y_{2}$ of $\mathbb{V}(H)\cup \mathbb{E}(H)$, $\zeta_{M}(y_{1}|M_{G}^{m})\neq \zeta_{M}(y_{2}|M_{G}^{m})$, then $M_{G}^{m}$ is said to be a mixed metric generator (or shortly, MMG) for $H$. The mixed metric generator with minimum cardinality is termed as the mixed metric basis, and that cardinality is known as the mixed metric dimension of the graph $H$, and which is denoted by $mdim(H)$ or $\beta_{M}(H)$. For our gentle purpose, by MMG and MMD we denote mixed metric generator and mixed metric dimension, respectively. Now, like edge metric dimension and the metric dimension, for this graph invariant one can define that, set $L_{M}$ consisting of vertices of the graph, $H$ is said to be an independent mixed metric generator for $H$, if $L_{M}$ is both a MMG and independent.

In this study, we consider two important families of the plane graphs viz., the prism allied graph $D^{t}_{n}$ (\cite{fwr}, see Figure 1) and the Web graph $\mathbb{W}_{n}$ (\cite{ys}, see Figure 2) and we obtain their MMD. Recently, the metric dimension and the edge metric dimension of these two families of the plane graphs were computed. For the metric dimension of these families of plane graphs we have the following results:

\begin{thm}\cite{fwr}
Let $D^{t}_{n}$ be the Prism allied graph on $6n$ edges and $4n$ vertices. Then, for $n\geq6$, we have $\beta(D^{t}_{n})=3$.
\end{thm}

\begin{thm}\cite{ys}
Let $\mathbb{W}_{n}$ be the Web graph on $4n$ edges and $3n$ vertices. Then, for $n\geq3$, we have

\begin{eqnarray*}
   \beta(\mathbb{W}_{n}) = \begin{cases}
                               2,  & if\  n\  is \ odd;\\
                               3,  & otherwise
                                             \end{cases}
\end{eqnarray*}
\end{thm}
and regarding the edge metric dimension, we have

\begin{thm}\cite{fwr}
Let $D^{t}_{n}$ be the Prism allied graph on $6n$ edges and $4n$ vertices. Then, for $n\geq3$, we have
\begin{eqnarray*}
   \beta_{E}(D^{t}_{n}) = \begin{cases}
                               4,  & if\  n=3,4,;\\
                                \left\lceil \frac{n}{2}\right\rceil+1,  & otherwise              .
                                             \end{cases}
\end{eqnarray*}
\end{thm}

\begin{thm}\cite{ys}
Let $\mathbb{W}_{n}$ be the Web graph on $4n$ edges and $3n$ vertices. Then, for $n\geq3$, we have $\beta_{E}(\mathbb{W}_{n})=3$.
\end{thm}

Throughout this article, all vertex indices are taken to be modulo $n$. The present paper is organized as follows: \par

In section 2, we study the MMD of the Prism allied graph $D^{t}_{n}$, when the MMG $M_{G}^{m}$ is independent (see Figures 1). In section 3, we study the MMD of the Web graph $\mathbb{W}_{n}$, when the MMG $M_{G}^{m}$ is independent (see Figures 2), and in our last section, we conclude our results and findings regarding these two important families of the plane graphs.

\section{Mixed Resolvability of the Prism Allied Graph $\mathbb{D}^{t}_{n}$}

The Prism allied graph $\mathbb{D}^{t}_{n}$ \cite{fwr} has vertex set of cardinality $4n$ and an edge set of cardinality $6n$, indicated by $\mathbb{V}(\mathbb{D}^{t}_{n})$ and $\mathbb{E}(\mathbb{D}^{t}_{n})$ respectively, where $\mathbb{V}(\mathbb{D}^{t}_{n})=\{p_{\pounds},q_{\pounds}, r_{\pounds},s_{\pounds}|1\leq \pounds \leq n\}$ and $\mathbb{E}(\mathbb{D}^{t}_{n})=\{p_{\pounds}q_{\pounds},p_{\pounds}p_{\pounds+1}, q_{\pounds}q_{\pounds+1}, r_{\pounds}q_{\pounds},r_{\pounds}q_{\pounds+1}, r_{\pounds}s_{\pounds}|1\leq \pounds \leq n\}$. It comprises of $n$ $3$-sided faces, $n$ pendant edges, $n$ $4$-sided faces, and an $n$-sided face (see Figure 1). The graph $\mathbb{D}^{t}_{n}$ is allied to the Prism graph $\mathbb{D}_{n}$ in the sense that, it can be acquired from the Prism graph by including new vertices $\{r_{\pounds},s_{\pounds}|1\leq \pounds \leq n\}$ and edges $\{r_{\pounds}q_{\pounds},r_{\pounds}q_{\pounds+1}, r_{\pounds}s_{\pounds}|1\leq \pounds \leq n\}$ in $\mathbb{D}_{n}$ as follows:
\begin{itemize}
  \item Placing new vertices $r_{\pounds}$, between the edges $q_{\pounds}q_{\pounds+1}$ $(1\leq \pounds \leq n)$.
  \item Again join the vertices $q_{\pounds}$ and $q_{\pounds+1}$.
  \item Join the vertices $r_{\pounds}$ and $s_{\pounds}$, in order to obtain the $n$ pendant edges.
\end{itemize}

\begin{center}
  \begin{figure}[h!]
  \centering
  \includegraphics[width=3in]{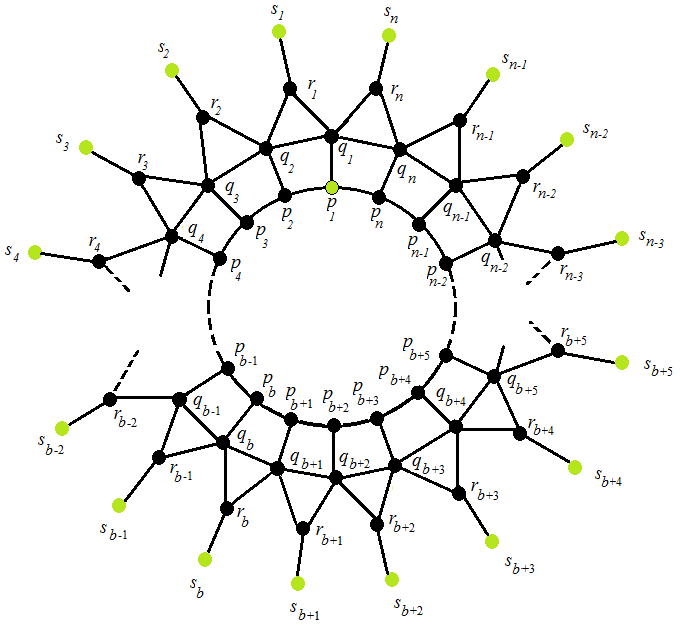}
  \caption{The Prism Allied Graph $\mathbb{D}^{t}_{n}$, for $n\geq4$.}\label{p1}
\end{figure}
\end{center}

For our smooth purpose, we refer to the cycle brought forth by the arrangement of vertices $\{q_{\pounds}:1 \leqslant \pounds \leqslant n\}$ and $\{p_{\pounds}:1 \leqslant \pounds \leqslant n\}$ in the graph, $\mathbb{D}^{t}_{n}$ as the $q$ and $p$-cycle respectively, the arrangement of vertices $\{r_{\pounds}:1 \leqslant \pounds \leqslant n\}$ and $\{s_{\pounds}:1 \leqslant \pounds \leqslant n\}$, in the graph, $\mathbb{D}^{t}_{n}$ as the set of outer and pendant vertices respectively. For our convenience, we consider $s_{1}=s_{n+1}$, $r_{1}=r_{n+1}$, $q_{1}=q_{n+1}$, and $p_{1}=p_{n+1}$. In the present working section, we obtain that the least possible cardinality for the MMG $M_{G}^{m}$ of the Prism allied graph $\mathbb{D}^{t}_{n}$ is $n+1$. We also find that the MMG $M_{G}^{m}$ for the Prism allied graph $\mathbb{D}^{t}_{n}$ is independent. Now, in order to get the exact MMD of graph $\mathbb{D}^{t}_{n}$, we need the following three Lemmas:

\begin{lem}
The set of outer vertices $\{s_{\pounds}| 1\leq \pounds \leq n\}\subset M_{G}^{m}$, where $M_{G}^{m}$ is a MMG for the Prism allied graph $\mathbb{D}^{t}_{n}$.
\end{lem}

\begin{proof}
For the inconsistency, we suppose that the MMG $M_{G}^{m}$, does not contain at least one vertex from the set $\{s_{\pounds}| 1\leq \pounds \leq n\}$. Without loss of generality, we suppose that $s_{\pounds}\not\in M_{G}^{m}$, for any $\pounds$. At that point, we have $\Im_{M}(r_{\pounds}|M_{G}^{m})=\Im_{M}(r_{\pounds}s_{\pounds}|M_{G}^{m})$, $\Im_{M}(q_{\pounds}|M_{G}^{m})=\Im_{M}(r_{\pounds}q_{\pounds}|M_{G}^{m})$, and $\Im_{M}(q_{\pounds+1}|M_{G}^{m})=\Im_{M}(r_{\pounds}q_{\pounds+1}|M_{G}^{m})$, a contradiction.

\end{proof}

\begin{lem}
Let $P=\{p_{\pounds}| 1\leq \pounds \leq n\}$ and $M_{G}^{m}$ be any mixed resolving generator for the Prism allied graph $\mathbb{D}^{t}_{n}$. Then, $P\cap M_{G}^{m}\neq \emptyset$.
\end{lem}

\begin{proof}
  Suppose on the contrary that $P\cap M_{G}^{m}=\emptyset$ i.e., for any $\pounds$, $p_{\pounds}\not\in M_{G}^{m}$. Then, we have $\Im_{M}(q_{\pounds}|M_{G}^{m})=\Im_{M}(p_{\pounds}q_{\pounds}|M_{G}^{m})$, a contradiction.
\end{proof}

In the accompanying Lemma, we show that the cardinality of any mixed resolving generator for the Prism allied graph $\mathbb{D}^{t}_{n}$ is greater than or equals to $n+1$ i.e., $|M_{G}^{m}|\geq n+1$.

\begin{lem}
For the Prism allied graph $\mathbb{D}^{t}_{n}$ and $n\geq4$, we have $mdim(\mathbb{D}^{t}_{n})\geq n+1$.
\end{lem}

\begin{proof}
On contrary, we suppose that the cardinality of the mixed resolving generator $M_{G}^{m}$ of the Prism allied graph $\mathbb{D}^{t}_{n}$ is equals to $n$ i.e., $\beta_{M}(\mathbb{D}^{t}_{n})=n$. Then, on combining Lemma 1 and 2, we get contradiction as the cardinality of the set $\{s_{\pounds}| 1\leq \pounds \leq n\}$ is $n$. So, we must have $\beta_{M}(\mathbb{D}^{t}_{n})\geq n+1$.
\end{proof}

Now, we are ready to obtain the exact mixed metric dimension for the Prism allied graph $\mathbb{D}^{t}_{n}$. For this, we have the following important result:\\\\
\textbf{Theorem 5.}
{\it For the Prism allied graph $\mathbb{D}^{t}_{n}$, we have $mdim(\mathbb{D}^{t}_{n})=n+1$, $\forall$ $n\geq4$.}

\begin{proof}
By Lemma 3, we have $mdim(\mathbb{D}^{t}_{n})\geq n+1$. Now, in order to complete the proof of the theorem, we have to show that $mdim(\mathbb{D}^{t}_{n})\leq n+1$. For this, suppose $M_{G}^{m}= \{p_{1}, s_{1}, s_{2},..., s_{n-1}, s_{n}\} \subset \mathbb{V}(\mathbb{D}^{t}_{n})$ (for the location of these vertices, see Figure 1 (vertices in green color)). We will show that $M_{G}^{m}$ is the MMG for the Prism allied graph $\mathbb{D}^{t}_{n}$. By total enumeration, it can be easily checked that the set $M_{G}^{m}$ is the MMG for the Prism allied graph $\mathbb{D}^{t}_{n}$, when $n=4,5$. Now, for $n\geq6$, we consider the following two cases regarding the positive integer $n$ (i.e., when $n\equiv0(modn)$ and $n\equiv1(mod2)$).\\
\vspace{-0.4mm}

\textbf{Case(\rom{1})} $n\equiv0(mod2)$.\\
In this case, $n$ can be written as $n = 2\aleph$, where $\aleph \in \mathbb{N}$ and $\aleph \geq 3$. Let $M_{G}^{\ast} = \{p_{1}, s_{1},s_{2}, s_{\aleph+1}, s_{\aleph+2}\}\\ \subset \mathbb{V}(\mathbb{D}^{t}_{n})$. Now, to figure out that $M_{G}^{\ast}$ is the MMG for the Prism allied graph $\mathbb{D}^{t}_{n}$, we consign the mixed metric codes for each vertex and each edge of the graph $\mathbb{D}^{t}_{n}$ regarding $M_{G}^{\ast}$ ($b=\aleph$ in Figure 1 and 2). \\

Now, the mixed metric codes for the vertices $\{\upsilon=p_{\pounds}, q_{\pounds}, r_{\pounds}, s_{\pounds}| \pounds=1,2,3,...,n\}$ regarding the set $M_{G}^{\ast}$ are shown below in the following four tables respectively.
\vspace{-0.4mm}
\begin{center}
 \begin{tabular}{|m{16.0em}|m{22.0em}|}
 \hline
  $\Im_{M}(\upsilon |M_{G}^{\ast})$ & $M_{G}^{\ast} = \{p_{1}, s_{1},s_{2}, s_{\aleph+1}, s_{\aleph+2}\}$ \\
 \hline
 \vspace{-0.4mm}
 $\Im_{M}(p_{\pounds}|M_{G}^{\ast})$:($\pounds=1$)                       & $(\pounds-1,3,4,\aleph+2,\aleph+1)$ \\
 \hline
 $\Im_{M}(p_{\pounds}|M_{G}^{\ast})$:($\pounds=2$)                       & $(\pounds-1,\pounds+1,3,\aleph-\pounds+4,\aleph+2)$ \\
 \hline
 $\Im_{M}(p_{\pounds}|M_{G}^{\ast})$:($3 \leq \pounds \leq \aleph+1 $)& $(\pounds-1,\pounds+1,\pounds,\aleph-\pounds+4,\aleph-\pounds+5)$ \\
 \hline
 $\Im_{M}(p_{\pounds}|M_{G}^{\ast})$:($\pounds=\aleph+2$)             & $(2\aleph-\pounds+1,2\aleph-\pounds+4,\pounds,\pounds-\aleph+1, \aleph-\pounds+5)$ \\
 \hline
 $\Im_{M}(p_{\pounds}|M_{G}^{\ast})$:($\aleph+3 \leq \pounds \leq 2\aleph$)   & $(2\aleph-\pounds+1,2\aleph-\pounds+4,2\aleph-\pounds+5,\pounds-\aleph+1,\pounds-\aleph)$ \\
 \hline
 \end{tabular}
 \end{center}
\vspace{-0.4mm}
\begin{center}
 \begin{tabular}{|m{16.0em}|m{22.0em}|}
 \hline
  $\Im_{M}(\upsilon |M_{G}^{\ast})$ & $M_{G}^{\ast} = \{p_{1}, s_{1},s_{2}, s_{\aleph+1}, s_{\aleph+2}\}$ \\
 \hline
 \vspace{-0.4mm}
 $\Im_{M}(q_{\pounds}|M_{G}^{\ast})$:($\pounds=1$)                       & $(\pounds,2,3,\aleph+1,\aleph)$ \\
 \hline
 $\Im_{M}(q_{\pounds}|M_{G}^{\ast})$:($\pounds=2$)                       & $(\pounds,\pounds,2,\aleph-\pounds+3,\aleph+1)$ \\
 \hline
 $\Im_{M}(q_{\pounds}|M_{G}^{\ast})$:($3 \leq \pounds \leq \aleph+1$) & $(\pounds,\pounds,\pounds-1,\aleph-\pounds+3,\aleph-\pounds+4)$ \\
 \hline
 $\Im_{M}(q_{\pounds}|M_{G}^{\ast})$:($\pounds=\aleph+2$)             & $(2\aleph-\pounds+2,2\aleph-\pounds+3,\pounds-1,\pounds-\aleph, \aleph-\pounds+4)$ \\
 \hline
 $\Im_{M}(q_{\pounds}|M_{G}^{\ast})$:($\aleph+3 \leq \pounds \leq 2\aleph$)   & $(2\aleph-\pounds+2,2\aleph-\pounds+3,2\aleph-\pounds+4,\pounds-\aleph,\pounds-\aleph-1)$ \\
 \hline
 \end{tabular}
 \end{center}
\vspace{-0.4mm}
\begin{center}
 \begin{tabular}{{|m{16.0em}|m{22.0em}|}}
 \hline
  $\Im_{M}(\upsilon |M_{G}^{\ast})$ & $M_{G}^{\ast} = \{p_{1}, s_{1},s_{2}, s_{\aleph+1}, s_{\aleph+2}\}$ \\
 \hline
 \vspace{-0.4mm}
 $\Im_{M}(r_{\pounds}|M_{G}^{\ast})$: ($\pounds=1$)                       & $(\pounds+1,1,3,\aleph-\pounds+3,\aleph+1)$ \\
 \hline
 $\Im_{M}(r_{\pounds}|M_{G}^{\ast})$: ($\pounds=2$)                       & $(\pounds+1,\pounds+1,1,\aleph-\pounds+3,\aleph-\pounds+4)$ \\
 \hline
 $\Im_{M}(r_{\pounds}|M_{G}^{\ast})$: ($3 \leq \pounds \leq \aleph$)   & $(\pounds+1,\pounds+1,\pounds,\aleph-\pounds+3,\aleph-\pounds+4)$ \\
 \hline
 $\Im_{M}(r_{\pounds}|M_{G}^{\ast})$: ($\pounds=\aleph+1$)             & $(2\aleph-\pounds+2,\pounds+1,\pounds,1,\aleph-\pounds+4)$ \\
 \hline
 $\Im_{M}(r_{\pounds}|M_{G}^{\ast})$: ($\pounds=\aleph+2$)             & $(2\aleph-\pounds+2,2\aleph-\pounds+3,\pounds,\pounds-\aleph+1,1)$ \\
 \hline
 $\Im_{M}(r_{\pounds}|M_{G}^{\ast})$: ($\aleph+3 \leq \pounds \leq 2\aleph$)   & $(2\aleph-\pounds+2,2\aleph-\pounds+3,2\aleph-\pounds+4,\pounds-\aleph+1,\pounds-\aleph)$ \\
 \hline
 \end{tabular}
 \end{center}
 \vspace{-0.4mm}
\begin{center}
 \begin{tabular}{{|m{16.0em}|m{22.0em}|}}
 \hline
  $\Im_{M}(\upsilon |M_{G}^{\ast})$ & $M_{G}^{\ast} = \{p_{1}, s_{1},s_{2}, s_{\aleph+1}, s_{\aleph+2}\}$ \\
 \hline
 \vspace{-0.4mm}
 $\Im_{M}(s_{\pounds}|M_{G}^{\ast})$: ($\pounds=1$)                       & $(\pounds+2,0,4,\aleph-\pounds+4,\aleph+2)$ \\
 \hline
 $\Im_{M}(s_{\pounds}|M_{G}^{\ast})$: ($\pounds=2$)                       & $(\pounds+2,\pounds+2,0,\aleph-\pounds+4,\aleph-\pounds+5)$ \\
 \hline
 $\Im_{M}(s_{\pounds}|M_{G}^{\ast})$: ($3 \leq \pounds \leq \aleph$)   & $(\pounds+2,\pounds+2,\pounds+1,\aleph-\pounds+4,\aleph-\pounds+5)$ \\
 \hline
 $\Im_{M}(s_{\pounds}|M_{G}^{\ast})$: ($\pounds=\aleph+1$)             & $(2\aleph-\pounds+3,2\aleph-\pounds+4,\pounds+1,0,\aleph-\pounds+5)$ \\
 \hline
 $\Im_{M}(s_{\pounds}|M_{G}^{\ast})$: ($\pounds=\aleph+2$)             & $(2\aleph-\pounds+3,2\aleph-\pounds+4,2\aleph-\pounds+5,\pounds-\aleph+2,0)$ \\
 \hline
 $\Im_{M}(s_{\pounds}|M_{G}^{\ast})$: ($\aleph+3 \leq \pounds \leq 2\aleph$)   & $(2\aleph-\pounds+3,2\aleph-\pounds+4,2\aleph-\pounds+5,\pounds-\aleph+2,\pounds-\aleph+1)$ \\
 \hline
 \end{tabular}
 \end{center}
\vspace{-0.4mm}
and the mixed metric codes for the edges $\{\epsilon=p_{\pounds}p_{\pounds+1}, p_{\pounds}q_{\pounds}, q_{\pounds}q_{\pounds+1}, q_{\pounds}r_{\pounds}, r_{\pounds}q_{\pounds+1}, r_{\pounds}s_{\pounds}| \pounds=1,2,3,...,n\}$ regarding the set $M_{G}^{\ast}$ are shown in the tables below, respectively.
\vspace{-0.4mm}
\begin{center}
 \begin{tabular}{{|m{16.0em}|m{22.0em}|}}
 \hline
  $\Im_{M}(\epsilon |M_{G}^{\ast})$ & $M_{G}^{\ast} = \{p_{1}, s_{1},s_{2}, s_{\aleph+1}, s_{\aleph+2}\}$ \\
 \hline
 \vspace{-0.4mm}
 $\Im_{M}(p_{\pounds}p_{\pounds+1}|M_{G}^{\ast})$: ($ \pounds=1$)                    & $(\pounds-1,3,3,\aleph-\pounds+3,\aleph+1)$ \\
 \hline
 $\Im_{M}(p_{\pounds}p_{\pounds+1}|M_{G}^{\ast})$: ($\pounds=2$)                     & $(\pounds-1,\pounds+1,3,\aleph-\pounds+3,\aleph-\pounds+4)$ \\
 \hline
 $\Im_{M}(p_{\pounds}p_{\pounds+1}|M_{G}^{\ast})$: ($3 \leq \pounds \leq \aleph$) & $(\pounds-1,\pounds+1,\pounds,\aleph-\pounds+3,\aleph-\pounds+4)$ \\
 \hline
 $\Im_{M}(p_{\pounds}p_{\pounds+1}|M_{G}^{\ast})$: ($\pounds=\aleph+1$)        & $(2\aleph-\pounds,2\aleph-\pounds+3,\pounds,3,\aleph-\pounds+4)$ \\
 \hline
 $\Im_{M}(p_{\pounds}p_{\pounds+1}|M_{G}^{\ast})$: ($\pounds=\aleph+2$)     & $(2\aleph-\pounds,2\aleph-\pounds+3,2\aleph-\pounds+4,\pounds-\aleph+1,3)$ \\
 \hline
 $\Im_{M}(p_{\pounds}p_{\pounds+1}|M_{G}^{\ast})$: ($\aleph+3 \leq \pounds \leq 2\aleph$)     & $(2\aleph-\pounds,2\aleph-\pounds+3,2\aleph-\pounds+4,\pounds-\aleph+1,\pounds-\aleph)$ \\
 \hline
 \end{tabular}
 \end{center}
\vspace{-0.4mm}
\begin{center}
 \begin{tabular}{{|m{16.0em}|m{22.0em}|}}
 \hline
  $\Im_{M}(\epsilon |M_{G}^{\ast})$ & $M_{G}^{\ast} = \{p_{1}, s_{1},s_{2}, s_{\aleph+1}, s_{\aleph+2}\}$ \\
 \hline
 \vspace{-0.4mm}
 $\Im_{M}(p_{\pounds}q_{\pounds}|M_{G}^{\ast})$: ($ \pounds=1$)                    & $(\pounds-1,2,3,\aleph+1,\aleph)$ \\
 \hline
 $\Im_{M}(p_{\pounds}q_{\pounds}|M_{G}^{\ast})$: ($\pounds=2$)                     & $(\pounds-1,\pounds,2,\aleph-\pounds+3,\aleph+1)$ \\
 \hline
 $\Im_{M}(p_{\pounds}q_{\pounds}|M_{G}^{\ast})$: ($3 \leq \pounds \leq \aleph+1$) & $(\pounds-1,\pounds,\pounds-1,\aleph-\pounds+3,\aleph-\pounds+4)$ \\
 \hline
 $\Im_{M}(p_{\pounds}q_{\pounds}|M_{G}^{\ast})$: ($\pounds=\aleph+2$)     & $(2\aleph-\pounds+1,2\aleph-\pounds+3,\pounds-1,\pounds-\aleph,\aleph-\pounds+4)$ \\
 \hline
 $\Im_{M}(p_{\pounds}q_{\pounds}|M_{G}^{\ast})$: ($\aleph+3 \leq \pounds \leq 2\aleph$)     & $(2\aleph-\pounds+1,2\aleph-\pounds+3,2\aleph-\pounds+4,\pounds-\aleph,\pounds-\aleph-1)$ \\
 \hline
 \end{tabular}
 \end{center}
\vspace{-0.4mm}
\begin{center}
 \begin{tabular}{{|m{16.0em}|m{22.0em}|}}
 \hline
  $\Im_{M}(\epsilon |M_{G}^{\ast})$ & $M_{G}^{\ast} = \{p_{1}, s_{1},s_{2}, s_{\aleph+1}, s_{\aleph+2}\}$ \\
 \hline
 \vspace{-0.4mm}
 $\Im_{M}(q_{\pounds}q_{\pounds+1}|M_{G}^{\ast})$: ($ \pounds=1$)                    & $(\pounds,2,2,\aleph-\pounds+2,\aleph)$ \\
 \hline
 $\Im_{M}(q_{\pounds}q_{\pounds+1}|M_{G}^{\ast})$: ($\pounds=2$)                     & $(\pounds,\pounds,2,\aleph-\pounds+2,\aleph-\pounds+3)$ \\
 \hline
 $\Im_{M}(q_{\pounds}q_{\pounds+1}|M_{G}^{\ast})$: ($3 \leq \pounds \leq \aleph$) & $(\pounds,\pounds,\pounds-1,\aleph-\pounds+2,\aleph-\pounds+3)$ \\
 \hline
 $\Im_{M}(q_{\pounds}q_{\pounds+1}|M_{G}^{\ast})$: ($\pounds=\aleph+1$)     & $(2\aleph-\pounds+1,2\aleph-\pounds+2,\pounds-1,2,\aleph-\pounds+3)$ \\
 \hline
 $\Im_{M}(q_{\pounds}q_{\pounds+1}|M_{G}^{\ast})$: ($\pounds=\aleph+2$)     & $(2\aleph-\pounds+1,2\aleph-\pounds+2,2\aleph-\pounds+3,\pounds-\aleph,2)$ \\
 \hline
 $\Im_{M}(q_{\pounds}q_{\pounds+1}|M_{G}^{\ast})$: ($\aleph+3 \leq \pounds \leq 2\aleph$)     & $(2\aleph-\pounds+1,2\aleph-\pounds+2,2\aleph-\pounds+3,\pounds-\aleph,\pounds-\aleph-1)$ \\
 \hline
 \end{tabular}
 \end{center}
\vspace{-0.4mm}
\begin{center}
 \begin{tabular}{{|m{16.0em}|m{22.0em}|}}
 \hline
  $\Im_{M}(\epsilon |M_{G}^{\ast})$ & $M_{G}^{\ast} = \{p_{1}, s_{1},s_{2}, s_{\aleph+1}, s_{\aleph+2}\}$ \\
 \hline
 \vspace{-0.4mm}
 $\Im_{M}(q_{\pounds}r_{\pounds}|M_{G}^{\ast})$: ($ \pounds=1$)                    & $(\pounds,1,3,\aleph+1,\aleph)$ \\
 \hline
 $\Im_{M}(q_{\pounds}r_{\pounds}|M_{G}^{\ast})$: ($\pounds=2$)                     & $(\pounds,\pounds,3,\aleph-\pounds+2,\aleph-\pounds+4)$ \\
 \hline
 $\Im_{M}(q_{\pounds}r_{\pounds}|M_{G}^{\ast})$: ($3 \leq \pounds \leq \aleph$) & $(\pounds,\pounds,\pounds-1,\aleph-\pounds+3,\aleph-\pounds+4)$ \\
 \hline
 $\Im_{M}(q_{\pounds}r_{\pounds}|M_{G}^{\ast})$: ($\pounds=\aleph+1$)     & $(\pounds,\pounds,\pounds-1,1,\aleph-\pounds+4)$ \\
 \hline
 $\Im_{M}(q_{\pounds}r_{\pounds}|M_{G}^{\ast})$: ($\pounds=\aleph+2$)     & $(2\aleph-\pounds+2,2\aleph-\pounds+3,\pounds-1,\pounds-\aleph,1)$ \\
 \hline
 $\Im_{M}(q_{\pounds}r_{\pounds}|M_{G}^{\ast})$: ($\aleph+3 \leq \pounds \leq 2\aleph$)     & $(2\aleph-\pounds+2,2\aleph-\pounds+3,2\aleph-\pounds+4,\pounds-\aleph,\pounds-\aleph-1)$ \\
 \hline
 \end{tabular}
 \end{center}
 \vspace{-0.4mm}
\begin{center}
 \begin{tabular}{{|m{16.0em}|m{22.0em}|}}
 \hline
  $\Im_{M}(\epsilon |M_{G}^{\ast})$ & $M_{G}^{\ast} = \{p_{1}, s_{1},s_{2}, s_{\aleph+1}, s_{\aleph+2}\}$ \\
 \hline
 \vspace{-0.4mm}
 $\Im_{M}(r_{\pounds}q_{\pounds+1}|M_{G}^{\ast})$: ($ \pounds=1$)                    & $(\pounds+1,1,2,\aleph-\pounds+2,\aleph+1)$ \\
 \hline
 $\Im_{M}(r_{\pounds}q_{\pounds+1}|M_{G}^{\ast})$: ($\pounds=2$)                     & $(\pounds+1,\pounds+1,1,\aleph-\pounds+2,\aleph-\pounds+3)$ \\
 \hline
 $\Im_{M}(r_{\pounds}q_{\pounds+1}|M_{G}^{\ast})$: ($3 \leq \pounds \leq \aleph$) & $(\pounds+1,\pounds+1,\pounds,\aleph-\pounds+2,\aleph-\pounds+3)$ \\
 \hline
 $\Im_{M}(r_{\pounds}q_{\pounds+1}|M_{G}^{\ast})$: ($\pounds=\aleph+1$)     & $(2\aleph-\pounds+1,2\aleph-\pounds+2,\pounds,1,\aleph-\pounds+3)$ \\
 \hline
 $\Im_{M}(r_{\pounds}q_{\pounds+1}|M_{G}^{\ast})$: ($\pounds=\aleph+2$)     & $(2\aleph-\pounds+1,2\aleph-\pounds+2,2\aleph-\pounds+3,\pounds-\aleph+1,1)$ \\
 \hline
 $\Im_{M}(r_{\pounds}q_{\pounds+1}|M_{G}^{\ast})$: ($\aleph+3 \leq \pounds \leq 2\aleph$)     & $(2\aleph-\pounds+1,2\aleph-\pounds+2,2\aleph-\pounds+3,\pounds-\aleph+1,\pounds-\aleph)$ \\
 \hline
 \end{tabular}
 \end{center}
 \vspace{-0.4mm}
\begin{center}
 \begin{tabular}{{|m{16.0em}|m{22.0em}|}}
 \hline
  $\Im_{M}(\epsilon |M_{G}^{\ast})$ & $M_{G}^{\ast} = \{p_{1}, s_{1},s_{2}, s_{\aleph+1}, s_{\aleph+2}\}$ \\
 \hline
 \vspace{-0.4mm}
 $\Im_{M}(r_{\pounds}s_{\pounds}|M_{G}^{\ast})$: ($ \pounds=1$)                    & $(\pounds+1,0,3,\aleph-\pounds+3,\aleph+1)$ \\
 \hline
 $\Im_{M}(r_{\pounds}s_{\pounds}|M_{G}^{\ast})$: ($\pounds=2$)                     & $(\pounds+1,\pounds+1,0,\aleph-\pounds+3,\aleph-\pounds+4)$ \\
 \hline
 $\Im_{M}(r_{\pounds}s_{\pounds}|M_{G}^{\ast})$: ($3 \leq \pounds \leq \aleph$) & $(\pounds+1,\pounds+1,\pounds,\aleph-\pounds+3,\aleph-\pounds+4)$ \\
 \hline
 $\Im_{M}(r_{\pounds}s_{\pounds}|M_{G}^{\ast})$: ($\pounds=\aleph+1$)     & $(2\aleph-\pounds+2,2\aleph-\pounds+3,\pounds,0,\aleph-\pounds+4)$ \\
 \hline
 $\Im_{M}(r_{\pounds}s_{\pounds}|M_{G}^{\ast})$: ($\pounds=\aleph+2$)     & $(2\aleph-\pounds+2,2\aleph-\pounds+3,2\aleph-\pounds+4,\pounds-\aleph+1,0)$ \\
 \hline
 $\Im_{M}(r_{\pounds}s_{\pounds}|M_{G}^{\ast})$: ($\aleph+3 \leq \pounds \leq 2\aleph$)     & $(2\aleph-\pounds+2,2\aleph-\pounds+3,2\aleph-\pounds+4,\pounds-\aleph+1,\pounds-\aleph)$ \\
 \hline
 \end{tabular}
 \end{center}

Now, from these mixed metric codes of the edges and the vertices of the Prism allied graph $\mathbb{D}^{t}_{n}$ concerning the set $M_{G}^{\ast}$, we ascertain that for $1\leq \pounds \leq n$ and $\pounds\neq 1,2, \aleph+1, \aleph+2$, $\Im_{M}(q_{\pounds}|M_{G}^{\ast})=\Im_{M}(r_{\pounds}q_{\pounds}|M_{G}^{\ast})$, $\Im_{M}(q_{\pounds+1}|M_{G}^{\ast})=\Im_{M}(r_{\pounds}q_{\pounds+1}|M_{G}^{\ast})$, and $\Im_{M}(r_{\pounds}|M_{G}^{\ast})=\Im_{M}(r_{\pounds}s_{\pounds}|M_{G}^{\ast})$. For the remaining mixed metric edges and vertices codes in $\mathbb{D}^{t}_{n}$, we find no two vertices or edges with the same mixed metric codes. For $\pounds=3,4,...,\aleph-1,\aleph,\aleph+2,\aleph+3,...,n$, we see that $\Im_{M}(q_{\pounds}|M_{G}^{\ast}\cup\{s_{\pounds}\})\neq\Im_{M}(r_{\pounds}q_{\pounds}|M_{G}^{\ast}\cup\{s_{\pounds}\})$, $\Im_{M}(q_{\pounds+1}|M_{G}^{\ast}\cup\{s_{\pounds}\})\neq\Im_{M}(r_{\pounds}q_{\pounds+1}|M_{G}^{\ast}\cup\{s_{\pounds}\})$, and $\Im_{M}(r_{\pounds}|M_{G}^{\ast}\cup\{s_{\pounds}\})\neq\Im_{M}(r_{\pounds}s_{\pounds}|M_{G}^{\ast}\cup\{s_{\pounds}\})$. From this, we obtain $\Im_{M}(q_{\pounds}|M_{G}^{m})\neq\Im_{M}(r_{\pounds}q_{\pounds}|M_{G}^{m})$, $\Im_{M}(q_{\pounds+1}|M_{G}^{m})\neq\Im_{M}(r_{\pounds}q_{\pounds+1}|M_{G}^{m})$, and $\Im_{M}(r_{\pounds}|M_{G}^{m})\neq\Im_{M}(r_{\pounds}s_{\pounds}|M_{G}^{m})$, for any $1\leq \pounds \leq n$ and so $|M_{G}^{m}|\leq n+1$, suggesting that $mdim(\mathbb{D}^{t}_{n})=n+1$ in this case.\\

\textbf{Case(\rom{2})} $n\equiv1(mod 2)$.\\
In this case, $n$ can be written as $n = 2\aleph+1$, where $\aleph \in \mathbb{N}$ and $\aleph \geq 3$. Let $M_{G}^{\ast} = \{p_{1}, s_{1},s_{2}, s_{\aleph+1},\\ s_{\aleph+2}\} \subset \mathbb{V}(\mathbb{D}^{t}_{n})$. Now, to figure out that $M_{G}^{\ast}$ is the MMG for the Prism allied graph $\mathbb{D}^{t}_{n}$, we consign the mixed metric codes for each vertex and each edge of the graph $\mathbb{D}^{t}_{n}$ regarding $M_{G}^{\ast}$. \\

Now, the mixed metric codes for the vertices $\{\upsilon=p_{\pounds}, q_{\pounds}, r_{\pounds}, s_{\pounds}| \pounds=1,2,3,...,n\}$ regarding the set $M_{G}^{\ast}$ are shown below in the following four tables respectively.
\vspace{-0.4mm}
\begin{center}
 \begin{tabular}{{|m{16.0em}|m{22.0em}|}}
 \hline
  $\Im_{M}(\upsilon |M_{G}^{\ast})$ & $M_{G}^{\ast} = \{p_{1}, s_{1},s_{2}, s_{\aleph+1}, s_{\aleph+2}\}$ \\
 \hline
 \vspace{-0.4mm}
 $\Im_{M}(p_{\pounds}|M_{G}^{\ast})$: ($\pounds=1$)                       & $(\pounds-1,3,4,\aleph-\pounds+4,\aleph+2)$ \\
 \hline
 $\Im_{M}(p_{\pounds}|M_{G}^{\ast})$: ($\pounds=2$)                       & $(\pounds-1,\pounds+1,3,\aleph-\pounds+4,\aleph-\pounds+5)$ \\
 \hline
 $\Im_{M}(p_{\pounds}|M_{G}^{\ast})$: ($3 \leq \pounds \leq \aleph+1 $)& $(\pounds-1,\pounds+1,\pounds,\aleph-\pounds+4,\aleph-\pounds+5)$ \\
 \hline
 $\Im_{M}(p_{\pounds}|M_{G}^{\ast})$: ($\pounds=\aleph+2$)             & $(2\aleph-\pounds+2,2\aleph-\pounds+5,\pounds,\pounds-\aleph+1,\aleph-\pounds+5)$ \\
 \hline
 $\Im_{M}(p_{\pounds}|M_{G}^{\ast})$:($\aleph+3 \leq \pounds \leq 2\aleph+1$)   & $(2\aleph-\pounds+2,2\aleph-\pounds+5,2\aleph-\pounds+6,\pounds-\aleph+1,\pounds-\aleph)$ \\
 \hline
 \end{tabular}
 \end{center}
\vspace{-0.4mm}
\begin{center}
 \begin{tabular}{{|m{16.0em}|m{22.0em}|}}
 \hline
  $\Im_{M}(\upsilon |M_{G}^{\ast})$ & $M_{G}^{\ast} = \{p_{1}, s_{1},s_{2}, s_{\aleph+1}, s_{\aleph+2}\}$ \\
 \hline
 \vspace{-0.4mm}
 $\Im_{M}(q_{\pounds}|M_{G}^{\ast})$: ($\pounds=1$)                       & $(\pounds,2,3,\aleph-\pounds+3,\aleph+1)$ \\
 \hline
 $\Im_{M}(q_{\pounds}|M_{G}^{\ast})$: ($\pounds=2$)                       & $(\pounds,\pounds,2,\aleph-\pounds+3,\aleph-\pounds+4)$ \\
 \hline
 $\Im_{M}(q_{\pounds}|M_{G}^{\ast})$: ($3 \leq \pounds \leq \aleph+1$) & $(\pounds,\pounds,\pounds-1,\aleph-\pounds+3,\aleph-\pounds+4)$ \\
 \hline
 $\Im_{M}(q_{\pounds}|M_{G}^{\ast})$: ($\pounds=\aleph+2$)             & $(2\aleph-\pounds+3,2\aleph-\pounds+4,\pounds-1,\pounds-\aleph,\aleph-\pounds+4)$ \\
 \hline
 $\Im_{M}(q_{\pounds}|M_{G}^{\ast})$: ($\aleph+3 \leq \pounds \leq 2\aleph+1$)   & $(2\aleph-\pounds+3,2\aleph-\pounds+4,2\aleph-\pounds+5,\pounds-\aleph,\pounds-\aleph-1)$ \\
 \hline
 \end{tabular}
 \end{center}
\vspace{-0.4mm}
\begin{center}
 \begin{tabular}{{|m{16.0em}|m{22.0em}|}}
 \hline
  $\Im_{M}(\upsilon |M_{G}^{\ast})$ & $M_{G}^{\ast} = \{p_{1}, s_{1},s_{2}, s_{\aleph+1}, s_{\aleph+2}\}$ \\
 \hline
 \vspace{-0.4mm}
 $\Im_{M}(r_{\pounds}|M_{G}^{\ast})$: ($\pounds=1$)                    & $(\pounds+1,1,3,\aleph-\pounds+3,\aleph+2)$ \\
 \hline
 $\Im_{M}(r_{\pounds}|M_{G}^{\ast})$: ($\pounds=2$)                    & $(\pounds+1,\pounds+1,1,\aleph-\pounds+3,\aleph-\pounds+4)$ \\
 \hline
 $\Im_{M}(r_{\pounds}|M_{G}^{\ast})$: ($3 \leq \pounds \leq \aleph$)   & $(\pounds+1,\pounds+1,\pounds,\aleph-\pounds+3,\aleph-\pounds+4)$ \\
 \hline
 $\Im_{M}(r_{\pounds}|M_{G}^{\ast})$: ($\pounds=\aleph+1$)             & $(\pounds+1,\pounds+1,\pounds,1,\aleph-\pounds+4)$ \\
 \hline
 $\Im_{M}(r_{\pounds}|M_{G}^{\ast})$: ($\pounds=\aleph+2$)             & $(2\aleph-\pounds+3,2\aleph-\pounds+4,\pounds,\pounds-\aleph+1,1)$ \\
 \hline
 $\Im_{M}(r_{\pounds}|M_{G}^{\ast})$: ($\aleph+3 \leq \pounds \leq 2\aleph+1$)   & $(2\aleph-\pounds+3,2\aleph-\pounds+4,2\aleph-\pounds+5,\pounds-\aleph+1,\pounds-\aleph)$ \\
 \hline
 \end{tabular}
 \end{center}
 \vspace{-0.4mm}
\begin{center}
 \begin{tabular}{{|m{16.0em}|m{22.0em}|}}
 \hline
  $\Im_{M}(\upsilon |M_{G}^{\ast})$ & $M_{G}^{\ast} = \{p_{1}, s_{1},s_{2}, s_{\aleph+1}, s_{\aleph+2}\}$ \\
 \hline
 \vspace{-0.4mm}
 $\Im_{M}(s_{\pounds}|M_{G}^{\ast})$: ($\pounds=1$)                       & $(\pounds+2,0,4,\aleph-\pounds+4,\aleph+3)$ \\
 \hline
 $\Im_{M}(s_{\pounds}|M_{G}^{\ast})$: ($\pounds=2$)                       & $(\pounds+2,\pounds+2,0,\aleph-\pounds+4,\aleph-\pounds+5)$ \\
 \hline
 $\Im_{M}(s_{\pounds}|M_{G}^{\ast})$: ($3 \leq \pounds \leq \aleph$)   & $(\pounds+2,\pounds+2,\pounds+1,\aleph-\pounds+4,\aleph-\pounds+5)$ \\
 \hline
 $\Im_{M}(s_{\pounds}|M_{G}^{\ast})$: ($\pounds=\aleph+1$)             & $(\pounds+2,\pounds+2,\pounds+1,0,\aleph-\pounds+5)$ \\
 \hline
 $\Im_{M}(s_{\pounds}|M_{G}^{\ast})$: ($\pounds=\aleph+2$)             & $(2\aleph-\pounds+4,2\aleph-\pounds+5,\pounds+1,\pounds-\aleph+2,0)$ \\
 \hline
 $\Im_{M}(s_{\pounds}|M_{G}^{\ast})$: ($\aleph+3 \leq \pounds \leq 2\aleph+1$)   & $(2\aleph-\pounds+4,2\aleph-\pounds+5,2\aleph-\pounds+6,\pounds-\aleph+2,\pounds-\aleph+1)$ \\
 \hline
 \end{tabular}
 \end{center}
\vspace{-0.4mm}
and the mixed metric codes for the edges $\{\epsilon=p_{\pounds}p_{\pounds+1}, p_{\pounds}q_{\pounds}, q_{\pounds}q_{\pounds+1}, q_{\pounds}r_{\pounds}, r_{\pounds}q_{\pounds+1}, r_{\pounds}s_{\pounds}| \pounds=1,2,3,...,n\}$ regarding the set $M_{G}^{\ast}$ are shown in the tables below, respectively.
\vspace{-0.4mm}
\begin{center}
 \begin{tabular}{{|m{16.5em}|m{21.5em}|}}
 \hline
  $\Im_{M}(\epsilon |M_{G}^{\ast})$ & $M_{G}^{\ast} = \{p_{1}, s_{1},s_{2}, s_{\aleph+1}, s_{\aleph+2}\}$ \\
 \hline
 \vspace{-0.4mm}
 $\Im_{M}(p_{\pounds}p_{\pounds+1}|M_{G}^{\ast})$: ($ \pounds=1$)                    & $(\pounds-1,3,3,\aleph-\pounds+3,\aleph+2)$ \\
 \hline
 $\Im_{M}(p_{\pounds}p_{\pounds+1}|M_{G}^{\ast})$: ($\pounds=2$)                     & $(\pounds-1,\pounds+1,3,\aleph-\pounds+3,\aleph-\pounds+4)$ \\
 \hline
 $\Im_{M}(p_{\pounds}p_{\pounds+1}|M_{G}^{\ast})$: ($3 \leq \pounds \leq \aleph$) & $(\pounds-1,\pounds+1,\pounds,\aleph-\pounds+3,\aleph-\pounds+4)$ \\
 \hline
 $\Im_{M}(p_{\pounds}p_{\pounds+1}|M_{G}^{\ast})$: ($\pounds=\aleph+1$)        & $(2\aleph-\pounds+1,\pounds+1,\pounds,3,\aleph-\pounds+4)$ \\
 \hline
 $\Im_{M}(p_{\pounds}p_{\pounds+1}|M_{G}^{\ast})$: ($\pounds=\aleph+2$)     & $(2\aleph-\pounds+1,2\aleph-\pounds+4,\pounds,\pounds-\aleph+1,3)$ \\
 \hline
 $\Im_{M}(p_{\pounds}p_{\pounds+1}|M_{G}^{\ast})$:($\aleph+3 \leq \pounds \leq 2\aleph+1$)     & $(2\aleph-\pounds+1,2\aleph-\pounds+4,2\aleph-\pounds+5,\pounds-\aleph+1,\pounds-\aleph)$ \\
 \hline
 \end{tabular}
 \end{center}
\vspace{-0.4mm}
\begin{center}
 \begin{tabular}{{|m{16.0em}|m{22.0em}|}}
 \hline
  $\Im_{M}(\epsilon |M_{G}^{\ast})$ & $M_{G}^{\ast} = \{p_{1}, s_{1},s_{2}, s_{\aleph+1}, s_{\aleph+2}\}$ \\
 \hline
 \vspace{-0.4mm}
 $\Im_{M}(p_{\pounds}q_{\pounds}|M_{G}^{\ast})$: ($ \pounds=1$)                    & $(\pounds-1,2,3,\aleph-\pounds+3,\aleph+1)$ \\
 \hline
 $\Im_{M}(p_{\pounds}q_{\pounds}|M_{G}^{\ast})$: ($\pounds=2$)                     & $(\pounds-1,\pounds,2,\aleph-\pounds+3,\aleph-\pounds+4)$ \\
 \hline
 $\Im_{M}(p_{\pounds}q_{\pounds}|M_{G}^{\ast})$: ($3 \leq \pounds \leq \aleph+1$) & $(\pounds-1,\pounds,\pounds-1,\aleph-\pounds+3,\aleph-\pounds+4)$ \\
 \hline
 $\Im_{M}(p_{\pounds}q_{\pounds}|M_{G}^{\ast})$: ($\pounds=\aleph+2$)     & $(2\aleph-\pounds+2,2\aleph-\pounds+4,\pounds-1,\pounds-\aleph,\aleph-\pounds+4)$ \\
 \hline
 $\Im_{M}(p_{\pounds}q_{\pounds}|M_{G}^{\ast})$: ($\aleph+3 \leq \pounds \leq 2\aleph+1$)     & $(2\aleph-\pounds+2,2\aleph-\pounds+4,2\aleph-\pounds+5,\pounds-\aleph,\pounds-\aleph-1)$ \\
 \hline
 \end{tabular}
 \end{center}
\vspace{-0.4mm}
\begin{center}
 \begin{tabular}{{|m{16.5em}|m{21.5em}|}}
 \hline
  $\Im_{M}(\epsilon |M_{G}^{\ast})$ & $M_{G}^{\ast} = \{p_{1}, s_{1},s_{2}, s_{\aleph+1}, s_{\aleph+2}\}$ \\
 \hline
 \vspace{-0.4mm}
 $\Im_{M}(q_{\pounds}q_{\pounds+1}|M_{G}^{\ast})$: ($ \pounds=1$)                    & $(\pounds,2,2,\aleph-\pounds+2,\aleph+1)$ \\
 \hline
 $\Im_{M}(q_{\pounds}q_{\pounds+1}|M_{G}^{\ast})$: ($\pounds=2$)                     & $(\pounds,\pounds,2,\aleph-\pounds+2,\aleph-\pounds+3)$ \\
 \hline
 $\Im_{M}(q_{\pounds}q_{\pounds+1}|M_{G}^{\ast})$: ($3 \leq \pounds \leq \aleph$) & $(\pounds,\pounds,\pounds-1,\aleph-\pounds+2,\aleph-\pounds+3)$ \\
 \hline
 $\Im_{M}(q_{\pounds}q_{\pounds+1}|M_{G}^{\ast})$: ($\pounds=\aleph+1$)     & $(2\aleph-\pounds+2,\pounds,\pounds-1,2,\aleph-\pounds+3)$ \\
 \hline
 $\Im_{M}(q_{\pounds}q_{\pounds+1}|M_{G}^{\ast})$: ($\pounds=\aleph+2$)     & $(2\aleph-\pounds+2,2\aleph-\pounds+3,\pounds-1,\pounds-\aleph,2)$ \\
 \hline
 $\Im_{M}(q_{\pounds}q_{\pounds+1}|M_{G}^{\ast})$:($\aleph+3 \leq \pounds \leq 2\aleph+1$)     & $(2\aleph-\pounds+2,2\aleph-\pounds+3,2\aleph-\pounds+4,\pounds-\aleph,\pounds-\aleph-1)$ \\
 \hline
 \end{tabular}
 \end{center}
\vspace{-0.4mm}
\begin{center}
 \begin{tabular}{{|m{16.0em}|m{22.0em}|}}
 \hline
  $\Im_{M}(\epsilon |M_{G}^{\ast})$ & $M_{G}^{\ast} = \{p_{1}, s_{1},s_{2}, s_{\aleph+1}, s_{\aleph+2}\}$ \\
 \hline
 \vspace{-0.4mm}
 $\Im_{M}(q_{\pounds}r_{\pounds}|M_{G}^{\ast})$: ($ \pounds=1$)                    & $(\pounds,1,3,\aleph-\pounds+3,\aleph+1)$ \\
 \hline
 $\Im_{M}(q_{\pounds}r_{\pounds}|M_{G}^{\ast})$: ($\pounds=2$)                     & $(\pounds,\pounds,1,\aleph-\pounds+3,\aleph-\pounds+4)$ \\
 \hline
 $\Im_{M}(q_{\pounds}r_{\pounds}|M_{G}^{\ast})$: ($3 \leq \pounds \leq \aleph$) & $(\pounds,\pounds,\pounds-1,\aleph-\pounds+3,\aleph-\pounds+4)$ \\
 \hline
 $\Im_{M}(q_{\pounds}r_{\pounds}|M_{G}^{\ast})$: ($\pounds=\aleph+1$)     & $(\pounds,\pounds,\pounds-1,1,\aleph-\pounds+4)$ \\
 \hline
 $\Im_{M}(q_{\pounds}r_{\pounds}|M_{G}^{\ast})$: ($\pounds=\aleph+2$)     & $(2\aleph-\pounds+3,2\aleph-\pounds+4,\pounds-1,\pounds-\aleph,1)$ \\
 \hline
 $\Im_{M}(q_{\pounds}r_{\pounds}|M_{G}^{\ast})$: ($\aleph+3 \leq \pounds \leq 2\aleph+1$)     & $(2\aleph-\pounds+3,2\aleph-\pounds+4,2\aleph-\pounds+5,\pounds-\aleph,\pounds-\aleph-1)$ \\
 \hline
 \end{tabular}
 \end{center}
 \vspace{-0.4mm}
\begin{center}
 \begin{tabular}{{|m{16.5em}|m{21.5em}|}}
 \hline
  $\Im_{M}(\epsilon |M_{G}^{\ast})$ & $M_{G}^{\ast} = \{p_{1}, s_{1},s_{2}, s_{\aleph+1}, s_{\aleph+2}\}$ \\
 \hline
 \vspace{-0.4mm}
 $\Im_{M}(r_{\pounds}q_{\pounds+1}|M_{G}^{\ast})$: ($ \pounds=1$)                    & $(\pounds+1,1,2,\aleph-\pounds+2,\aleph-\pounds+3)$ \\
 \hline
 $\Im_{M}(r_{\pounds}q_{\pounds+1}|M_{G}^{\ast})$: ($\pounds=2$)                     & $(\pounds+1,\pounds+1,1,\aleph-\pounds+2,\aleph-\pounds+3)$ \\
 \hline
 $\Im_{M}(r_{\pounds}q_{\pounds+1}|M_{G}^{\ast})$: ($3 \leq \pounds \leq \aleph$) & $(\pounds+1,\pounds+1,\pounds,\aleph-\pounds+2,\aleph-\pounds+3)$ \\
 \hline
 $\Im_{M}(r_{\pounds}q_{\pounds+1}|M_{G}^{\ast})$: ($\pounds=\aleph+1$)     & $(2\aleph-\pounds+2,2\aleph-\pounds+3,\pounds,1,\aleph-\pounds+3)$ \\
 \hline
 $\Im_{M}(r_{\pounds}q_{\pounds+1}|M_{G}^{\ast})$: ($\pounds=\aleph+2$)     & $(2\aleph-\pounds+2,2\aleph-\pounds+3,2\aleph-\pounds+4,\pounds-\aleph+1,1)$ \\
 \hline
 $\Im_{M}(r_{\pounds}q_{\pounds+1}|M_{G}^{\ast})$:($\aleph+3 \leq \pounds \leq 2\aleph+1$)     & $(2\aleph-\pounds+2,2\aleph-\pounds+3,2\aleph-\pounds+4,\pounds-\aleph+1,\pounds-\aleph)$ \\
 \hline
 \end{tabular}
 \end{center}
 \vspace{-0.4mm}
\begin{center}
 \begin{tabular}{{|m{16.0em}|m{22.0em}|}}
 \hline
  $\Im_{M}(\epsilon |M_{G}^{\ast})$ & $M_{G}^{\ast} = \{p_{1}, s_{1},s_{2}, s_{\aleph+1}, s_{\aleph+2}\}$ \\
 \hline
 \vspace{-0.4mm}
 $\Im_{M}(r_{\pounds}s_{\pounds}|M_{G}^{\ast})$: ($ \pounds=1$)                    & $(\pounds+1,0,3,\aleph-\pounds+3,\aleph+2)$ \\
 \hline
 $\Im_{M}(r_{\pounds}s_{\pounds}|M_{G}^{\ast})$: ($\pounds=2$)                     & $(\pounds+1,\pounds+1,0,\aleph-\pounds+3,\aleph-\pounds+4)$ \\
 \hline
 $\Im_{M}(r_{\pounds}s_{\pounds}|M_{G}^{\ast})$: ($3 \leq \pounds \leq \aleph$) & $(\pounds+1,\pounds+1,\pounds,\aleph-\pounds+3,\aleph-\pounds+4)$ \\
 \hline
 $\Im_{M}(r_{\pounds}s_{\pounds}|M_{G}^{\ast})$: ($\pounds=\aleph+1$)     & $(\pounds+1,\pounds+1,\pounds,0,\aleph-\pounds+4)$ \\
 \hline
 $\Im_{M}(r_{\pounds}s_{\pounds}|M_{G}^{\ast})$: ($\pounds=\aleph+2$)     & $(2\aleph-\pounds+3,2\aleph-\pounds+4,\pounds,\pounds-\aleph+1,0)$ \\
 \hline
 $\Im_{M}(r_{\pounds}s_{\pounds}|M_{G}^{\ast})$: ($\aleph+3 \leq \pounds \leq 2\aleph+1$)     & $(2\aleph-\pounds+3,2\aleph-\pounds+4,2\aleph-\pounds+5,\pounds-\aleph+1,\pounds-\aleph)$ \\
 \hline
 \end{tabular}
 \end{center}
Now, from these mixed metric codes of the edges and the vertices of the Prism allied graph $\mathbb{D}^{t}_{n}$ concerning the set $M_{G}^{\ast}$, we ascertain that for $1\leq \pounds \leq n$ and $\pounds\neq 1,2, \aleph+1, \aleph+2$, $\Im_{M}(q_{\pounds}|M_{G}^{\ast})=\Im_{M}(r_{\pounds}q_{\pounds}|M_{G}^{\ast})$, $\Im_{M}(q_{\pounds+1}|M_{G}^{\ast})=\Im_{M}(r_{\pounds}q_{\pounds+1}|M_{G}^{\ast})$, and $\Im_{M}(r_{\pounds}|M_{G}^{\ast})=\Im_{M}(r_{\pounds}s_{\pounds}|M_{G}^{\ast})$. For the remaining mixed metric edges and vertices codes in $\mathbb{D}^{t}_{n}$, we find no two vertices or edges with the same mixed metric codes. For $\pounds=3,4,...,\aleph-1,\aleph,\aleph+2,\aleph+3,...,n$, we see that $\Im_{M}(q_{\pounds}|M_{G}^{\ast}\cup\{s_{\pounds}\})\neq\Im_{M}(r_{\pounds}q_{\pounds}|M_{G}^{\ast}\cup\{s_{\pounds}\})$, $\Im_{M}(q_{\pounds+1}|M_{G}^{\ast}\cup\{s_{\pounds}\})\neq\Im_{M}(r_{\pounds}q_{\pounds+1}|M_{G}^{\ast}\cup\{s_{\pounds}\})$, and $\Im_{M}(r_{\pounds}|M_{G}^{\ast}\cup\{s_{\pounds}\})\neq\Im_{M}(r_{\pounds}s_{\pounds}|M_{G}^{\ast}\cup\{s_{\pounds}\})$. From this, we obtain $\Im_{M}(q_{\pounds}|M_{G}^{m})\neq\Im_{M}(r_{\pounds}q_{\pounds}|M_{G}^{m})$, $\Im_{M}(q_{\pounds+1}|M_{G}^{m})\neq\Im_{M}(r_{\pounds}q_{\pounds+1}|M_{G}^{m})$, and $\Im_{M}(r_{\pounds}|M_{G}^{m})\neq\Im_{M}(r_{\pounds}s_{\pounds}|M_{G}^{m})$, for any $1\leq \pounds \leq n$ and so $|M_{G}|\leq n+1$, suggesting that $mdim(\mathbb{D}^{t}_{n})=n+1$ in this case also, which concludes the theorem.
\end{proof}
\hspace{-5.0mm}\textbf{Theorem 6.}
{\it The independent mixed metric number for the Prism allied graph $\mathbb{D}^{t}_{n}$, for $n\geq4$ is $n+1$.}
\begin{proof}
  For proof, refer to Theorem $5$.
\end{proof}

\section{Mixed Resolvability of the Web Graph $\mathbb{W}_{n}$}
The Web graph $\mathbb{W}_{n}$ \cite{ys} has a vertex set of cardinality $3n$ and an edge set of cardinality $4n$, indicated by $\mathbb{V}(\mathbb{W}_{n})$ and $\mathbb{E}(\mathbb{W}_{n})$ respectively, where $\mathbb{V}(\mathbb{W}_{n})=\{p_{\pounds},q_{\pounds}, r_{\pounds}|1\leq \pounds \leq n\}$ and $\mathbb{E}(\mathbb{W}_{n})=\{p_{\pounds}q_{\pounds},p_{\pounds}p_{\pounds+1}, q_{\pounds}q_{\pounds+1}, r_{\pounds}q_{\pounds}|1\leq \pounds \leq n\}$. It comprises of $n$ $4$-sided faces, $n$ pendant edges, and an $n$-sided face (see Figure 2). The Web graph $\mathbb{W}_{n}$ can also be obtained from the Prism graph $\mathbb{D}_{n}$ by simply including $n$ new pendant edges $q_{\pounds}r_{\pounds}$ $(1\leq \pounds \leq n)$.

\begin{center}
  \begin{figure}[h!]
  \centering
  \includegraphics[width=3in]{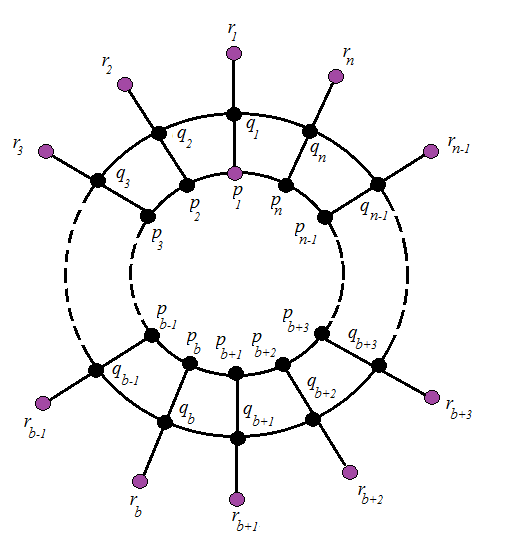}
  \caption{The Web Graph $\mathbb{W}_{n}$, for $n\geq4$.}\label{p1}
\end{figure}
\end{center}

For our smooth purpose, we refer to the cycle brought forth by the arrangement of vertices $\{q_{\pounds}:1 \leqslant \pounds \leqslant n\}$ and $\{p_{\pounds}:1 \leqslant \pounds \leqslant n\}$ in the graph, $\mathbb{W}_{n}$ as the $q$ and $p$-cycle respectively, the arrangement of vertices $\{r_{\pounds}:1 \leqslant \pounds \leqslant n\}$, in the graph, $\mathbb{W}_{n}$ as the set of pendant vertices respectively. For our convenience, we consider $r_{1}=r_{n+1}$, $q_{1}=q_{n+1}$, and $p_{1}=p_{n+1}$. In this working section, we obtain that the least possible cardinality for the MMG $M_{G}^{m}$ of the Web graph $\mathbb{W}_{n}$ is $n+1$. For this, we also see that the MMG $M_{G}^{m}$ for the Web graph $\mathbb{W}_{n}$ is independent. Now, in order to get the exact MMD of graph $\mathbb{W}_{n}$, we need the following Lemmas:

\begin{lem}
The set of outer vertices $\{r_{\pounds}| 1\leq \pounds \leq n\}\subset M_{G}^{m}$, where $M_{G}^{m}$ is a MMG for the Web graph $\mathbb{W}_{n}$.
\end{lem}

\begin{proof}
For the inconsistency, we suppose that the MMG $M_{G}^{m}$, does not contain at least one vertex from the set $\{r_{\pounds}| 1\leq \pounds \leq n\}$. Without loss of generality, we suppose that $r_{\pounds}\not\in M_{G}^{m}$, for any $\pounds$. At that point, we have $\Im_{M}(r_{\pounds}|M_{G}^{m})=\Im_{M}(r_{\pounds}q_{\pounds}|M_{G}^{m})$, a contradiction.

\end{proof}

\begin{lem}
Let $P=\{p_{\pounds}| 1\leq \pounds \leq n\}$ and $M_{G}^{m}$ be any MMG for the Web graph $\mathbb{W}_{n}$. Then, $P \cap M_{G}^{m}\neq \emptyset$.
\end{lem}

\begin{proof}
  Suppose on the contrary that $P\cap M_{G}^{m}=\emptyset$ i.e., for any $\pounds$, $p_{\pounds}\not\in M_{G}^{m}$. Then, we have $\Im_{M}(q_{\pounds}|M_{G}^{m})=\Im_{M}(p_{\pounds}q_{\pounds}|M_{G}^{m})$, a contradiction.
\end{proof}

In the accompanying Lemma, we show that the cardinality of any MMG for the Web graph $\mathbb{W}_{n}$ is greater than or equals to $n+1$ i.e., $|M_{G}^{m}|\geq n+1$.

\begin{lem}
For the Web graph $\mathbb{W}_{n}$ and $n\geq4$, we have $mdim(\mathbb{W}_{n})\geq n+1$.
\end{lem}

\begin{proof}
On contrary, we suppose that the cardinality of the MMG $M_{G}^{m}$ of the Web graph $\mathbb{W}_{n}$ is equals $n$ i.e., $\beta_{M}(\mathbb{W}_{n})=n$. Then, on combining Lemma 4 and 5, we get contradiction as the cardinality of the set $\{r_{\pounds}| 1\leq \pounds \leq n\}$ is $n$. So, we must have $\beta_{M}(\mathbb{W}_{n})\geq n+1$.
\end{proof}

Now, for the Web graph $\mathbb{W}_{n}$, we have the following important result regarding its MMD:\\\\
\textbf{Theorem 7.}
{\it For the Web graph $\mathbb{W}_{n}$, we have $mdim(\mathbb{W}_{n})=n+1$, $\forall$ $n\geq4$.}

\begin{proof}
By Lemma 6, we have $mdim(\mathbb{W}_{n})\geq n+1$. Now, in order to complete the proof of the theorem, we have to show that $mdim(\mathbb{W}_{n})\leq n+1$. For this, suppose $M_{G}^{m}= \{p_{1}, r_{1}, r_{2},..., r_{n-1}, r_{n}\} \subset \mathbb{V}(\mathbb{W}_{n})$ (for the location of these vertices, see Figure 2 (vertices in purple color)). We will show that $M_{G}^{m}$ is the mixed metric basis set for the Web graph $\mathbb{W}_{n}$. By total enumeration, it can be easily checked that the set $M_{G}^{m}$ is the mixed metric basis set for the Web graph $\mathbb{W}_{n}$, when $n=4,5$. Now, for $n\geq6$, we consider the following two cases regarding the positive integer $n$ (i.e., when $n\equiv0(modn)$ and $n\equiv1(mod2)$).\\
\vspace{-0.4mm}

\textbf{Case(\rom{1})} $n\equiv0(mod2)$.\\
In this case, $n$ can be written as $n = 2\aleph$, where $\aleph \in \mathbb{N}$ and $\aleph \geq 3$. Let $M_{G}^{\ast} = \{p_{1}, r_{1},r_{2}, r_{\aleph+1}, r_{\aleph+2}\}\\ \subset \mathbb{V}(\mathbb{W}_{n})$. Now, to figure out that $M_{G}^{\ast}$ is the MMG for the Web graph $\mathbb{W}_{n}$, we consign the mixed metric codes for each vertex and each edge of the graph $\mathbb{W}_{n}$ regarding $M_{G}^{\ast}$. \\

Now, the mixed metric codes for the vertices $\{\upsilon=p_{\pounds}, q_{\pounds}, r_{\pounds}| \pounds=1,2,3,...,n\}$ regarding the set $M_{G}^{\ast}$ are shown below in the following three tables respectively.
\vspace{-0.4mm}
\begin{center}
 \begin{tabular}{{|m{16.0em}|m{22.0em}|}}
 \hline
  $\Im_{M}(\upsilon |M_{G}^{\ast})$ & $M_{G}^{\ast} = \{p_{1}, r_{1},r_{2}, r_{\aleph+1}, r_{\aleph+2}\}$ \\
 \hline
 \vspace{-0.4mm}
 $\Im_{M}(p_{\pounds}|M_{G}^{\ast})$: ($\pounds=1$)                       & $(\pounds-1,\pounds+1,3,\aleph-\pounds+3,\aleph+1)$ \\
 \hline
 $\Im_{M}(p_{\pounds}|M_{G}^{\ast})$: ($2 \leq \pounds \leq \aleph+1$) & $(\pounds-1,\pounds+1,\pounds,\aleph-\pounds+3,\aleph-\pounds+4)$ \\
 \hline
 $\Im_{M}(p_{\pounds}|M_{G}^{\ast})$: ($\pounds=\aleph+2$)             & $(2\aleph-\pounds+1,2\aleph-\pounds+3,\pounds,\pounds-\aleph+1,\aleph-\pounds+4)$ \\
 \hline
 $\Im_{M}(p_{\pounds}|M_{G}^{\ast})$: ($\aleph+3 \leq \pounds \leq 2\aleph$)   & $(2\aleph-\pounds+1,2\aleph-\pounds+3,2\aleph-\pounds+4,\pounds-\aleph+1,\pounds-\aleph)$ \\
 \hline
 \end{tabular}
 \end{center}
\vspace{-0.4mm}
\begin{center}
 \begin{tabular}{{|m{16.0em}|m{22.0em}|}}
 \hline
  $\Im_{M}(\upsilon |M_{G}^{\ast})$ & $M_{G}^{\ast} = \{p_{1}, r_{1},r_{2}, r_{\aleph+1}, r_{\aleph+2}\}$ \\
 \hline
 \vspace{-0.4mm}
 $\Im_{M}(q_{\pounds}|M_{G}^{\ast})$: ($\pounds=1$)                       & $(\pounds,\pounds,2,\aleph-\pounds+2,\aleph)$ \\
 \hline
 $\Im_{M}(q_{\pounds}|M_{G}^{\ast})$: ($2 \leq \pounds \leq \aleph+1$) & $(\pounds,\pounds,\pounds-1,\aleph-\pounds+2,\aleph-\pounds+3)$ \\
 \hline
 $\Im_{M}(q_{\pounds}|M_{G}^{\ast})$: ($\pounds=\aleph+2$)             & $(2\aleph-\pounds+2,2\aleph-\pounds+2,\pounds-1,\pounds-\aleph, \aleph-\pounds+3)$ \\
 \hline
 $\Im_{M}(q_{\pounds}|M_{G}^{\ast})$: ($\aleph+3 \leq \pounds \leq 2\aleph$)   & $(2\aleph-\pounds+2,2\aleph-\pounds+2,2\aleph-\pounds+3,\pounds-\aleph,\pounds-\aleph-1)$ \\
 \hline
 \end{tabular}
 \end{center}
\vspace{-0.4mm}
\begin{center}
 \begin{tabular}{{|m{16.0em}|m{22.0em}|}}
 \hline
  $\Im_{M}(\upsilon |M_{G}^{\ast})$ & $M_{G}^{\ast} = \{p_{1}, r_{1},r_{2}, r_{\aleph+1}, r_{\aleph+2}\}$ \\
 \hline
 \vspace{-0.4mm}
 $\Im_{M}(r_{\pounds}|M_{G}^{\ast})$: ($\pounds=1$)                       & $(\pounds+1,0,3,\aleph-\pounds+3,\aleph+1)$ \\
 \hline
 $\Im_{M}(r_{\pounds}|M_{G}^{\ast})$: ($\pounds=2$)                       & $(\pounds+1,\pounds+1,0,\aleph-\pounds+3,\aleph-\pounds+4)$ \\
 \hline
 $\Im_{M}(r_{\pounds}|M_{G}^{\ast})$: ($3 \leq \pounds \leq \aleph$)   & $(\pounds+1,\pounds+1,\pounds,\aleph-\pounds+3,\aleph-\pounds+4)$ \\
 \hline
 $\Im_{M}(r_{\pounds}|M_{G}^{\ast})$: ($\pounds=\aleph+1$)             & $(\pounds+1,\pounds+1,\pounds,0,\aleph-\pounds+4)$ \\
 \hline
 $\Im_{M}(r_{\pounds}|M_{G}^{\ast})$: ($\pounds=\aleph+2$)             & $(2\aleph-\pounds+3,2\aleph-\pounds+3,\pounds,\pounds-\aleph+1,0)$ \\
 \hline
 $\Im_{M}(r_{\pounds}|M_{G}^{\ast})$: ($\aleph+3 \leq \pounds \leq 2\aleph$)   & $(2\aleph-\pounds+3,2\aleph-\pounds+3,2\aleph-\pounds+4,\pounds-\aleph+1,\pounds-\aleph)$ \\
 \hline
 \end{tabular}
 \end{center}
 \vspace{-0.4mm}
and the mixed metric codes for the edges $\{\epsilon=p_{\pounds}p_{\pounds+1}, p_{\pounds}q_{\pounds}, q_{\pounds}q_{\pounds+1}, q_{\pounds}r_{\pounds}| \pounds=1,2,3,...,n\}$ regarding the set $M_{G}^{\ast}$ are shown in the tables below, respectively.
\vspace{-0.4mm}
\begin{center}
 \begin{tabular}{{|m{16.0em}|m{22.0em}|}}
 \hline
  $\Im_{M}(\epsilon |M_{G}^{\ast})$ & $M_{G}^{\ast} = \{p_{1}, r_{1},r_{2}, r_{\aleph+1}, r_{\aleph+2}\}$ \\
 \hline
 \vspace{-0.4mm}
 $\Im_{M}(p_{\pounds}p_{\pounds+1}|M_{G}^{\ast})$: ($ \pounds=1$)                    & $(\pounds-1,\pounds+1,2,\aleph-\pounds+2,\aleph+1)$ \\
 \hline
 $\Im_{M}(p_{\pounds}p_{\pounds+1}|M_{G}^{\ast})$: ($2 \leq \pounds \leq \aleph$) & $(\pounds-1,\pounds+1,\pounds,\aleph-\pounds+2,\aleph-\pounds+3)$ \\
 \hline
 $\Im_{M}(p_{\pounds}p_{\pounds+1}|M_{G}^{\ast})$: ($\pounds=\aleph+1$)   & $(2\aleph-\pounds,2\aleph-\pounds+2,\pounds,\pounds-\aleph+1,\aleph-\pounds+3)$ \\
 \hline
 $\Im_{M}(p_{\pounds}p_{\pounds+1}|M_{G}^{\ast})$: ($\aleph+2 \leq \pounds \leq 2\aleph$)     & $(2\aleph-\pounds,2\aleph-\pounds+2,2\aleph-\pounds+3,\pounds-\aleph+1,\pounds-\aleph)$ \\
 \hline
 \end{tabular}
 \end{center}
\vspace{-0.4mm}
\begin{center}
 \begin{tabular}{{|m{16.0em}|m{22.0em}|}}
 \hline
  $\Im_{M}(\epsilon |M_{G}^{\ast})$ & $M_{G}^{\ast} = \{p_{1}, r_{1},r_{2}, r_{\aleph+1}, r_{\aleph+2}\}$ \\
 \hline
 \vspace{-0.4mm}
 $\Im_{M}(p_{\pounds}q_{\pounds}|M_{G}^{\ast})$: ($ \pounds=1$)                    & $(\pounds-1,\pounds,2,\aleph-\pounds+2,\aleph)$ \\
 \hline
 $\Im_{M}(p_{\pounds}q_{\pounds}|M_{G}^{\ast})$: ($3 \leq \pounds \leq \aleph+1$) & $(\pounds-1,\pounds,\pounds-1,\aleph-\pounds+2,\aleph-\pounds+3)$ \\
 \hline
 $\Im_{M}(p_{\pounds}q_{\pounds}|M_{G}^{\ast})$: ($\pounds=\aleph+2$)     & $(2\aleph-\pounds+1,2\aleph-\pounds+2,\pounds-1,\pounds-\aleph,\aleph-\pounds+3)$ \\
 \hline
 $\Im_{M}(p_{\pounds}q_{\pounds}|M_{G}^{\ast})$: ($\aleph+3 \leq \pounds \leq 2\aleph$)     & $(2\aleph-\pounds+1,2\aleph-\pounds+2,2\aleph-\pounds+3,\pounds-\aleph,\pounds-\aleph-1)$ \\
 \hline
 \end{tabular}
 \end{center}
\vspace{-0.4mm}
\begin{center}
 \begin{tabular}{{|m{16.0em}|m{22.0em}|}}
 \hline
  $\Im_{M}(\epsilon |M_{G}^{\ast})$ & $M_{G}^{\ast} = \{p_{1}, r_{1},r_{2}, r_{\aleph+1}, r_{\aleph+2}\}$ \\
 \hline
 \vspace{-0.4mm}
 $\Im_{M}(q_{\pounds}q_{\pounds+1}|M_{G}^{\ast})$: ($ \pounds=1$)                    & $(\pounds,\pounds,1,\aleph-\pounds+1,\aleph)$ \\
 \hline
 $\Im_{M}(q_{\pounds}q_{\pounds+1}|M_{G}^{\ast})$: ($2 \leq \pounds \leq \aleph$) & $(\pounds,\pounds,\pounds-1,\aleph-\pounds+1,\aleph-\pounds+2)$ \\
 \hline
 $\Im_{M}(q_{\pounds}q_{\pounds+1}|M_{G}^{\ast})$: ($\pounds=\aleph+1$)   & $(2\aleph-\pounds+1,2\aleph-\pounds+1,\pounds-1,\pounds-\aleph,\aleph-\pounds+2)$ \\
 \hline
 $\Im_{M}(q_{\pounds}q_{\pounds+1}|M_{G}^{\ast})$: ($\aleph+2 \leq \pounds \leq 2\aleph$)     & $(2\aleph-\pounds+1,2\aleph-\pounds+1,2\aleph-\pounds+2,\pounds-\aleph,\pounds-\aleph-1)$ \\
 \hline
 \end{tabular}
 \end{center}
\vspace{-0.4mm}
\begin{center}
 \begin{tabular}{{|m{16.0em}|m{22.0em}|}}
 \hline
  $\Im_{M}(\epsilon |M_{G}^{\ast})$ & $M_{G}^{\ast} = \{p_{1}, r_{1},r_{2}, r_{\aleph+1}, r_{\aleph+2}\}$ \\
 \hline
 \vspace{-0.4mm}
 $\Im_{M}(q_{\pounds}r_{\pounds}|M_{G}^{\ast})$: ($ \pounds=1$)                    & $(\pounds,0,2,\aleph-\pounds+2,\aleph)$ \\
 \hline
 $\Im_{M}(q_{\pounds}r_{\pounds}|M_{G}^{\ast})$: ($\pounds=2$)                     & $(\pounds,\pounds,0,\aleph-\pounds+2,\aleph-\pounds+3)$ \\
 \hline
 $\Im_{M}(q_{\pounds}r_{\pounds}|M_{G}^{\ast})$: ($3 \leq \pounds \leq \aleph$) & $(\pounds,\pounds,\pounds-1,\aleph-\pounds+2,\aleph-\pounds+3)$ \\
 \hline
 $\Im_{M}(q_{\pounds}r_{\pounds}|M_{G}^{\ast})$: ($\pounds=\aleph+1$)     & $(\pounds,\pounds,\pounds-1,0,\aleph-\pounds+3)$ \\
 \hline
 $\Im_{M}(q_{\pounds}r_{\pounds}|M_{G}^{\ast})$: ($\pounds=\aleph+2$)     & $(2\aleph-\pounds+2,2\aleph-\pounds+2,\pounds-1,\pounds-\aleph,0)$ \\
 \hline
 $\Im_{M}(q_{\pounds}r_{\pounds}|M_{G}^{\ast})$: ($\aleph+3 \leq \pounds \leq 2\aleph$)     & $(2\aleph-\pounds+2,2\aleph-\pounds+2,2\aleph-\pounds+3,\pounds-\aleph,\pounds-\aleph-1)$ \\
 \hline
 \end{tabular}
 \end{center}

Now, from these mixed metric codes of the edges and the vertices of the Web graph $\mathbb{W}_{n}$ concerning the set $M_{G}^{\ast}$, we ascertain that for $1\leq \pounds \leq n$ and $\pounds\neq 1,2, \aleph+1, \aleph+2$, $\Im_{M}(q_{\pounds}|M_{G}^{\ast})=\Im_{M}(r_{\pounds}q_{\pounds}|M_{G}^{\ast})$. For the remaining mixed metric edges and vertices codes in $\mathbb{W}_{n}$, we find no two vertices or edges with the same mixed metric codes. For $\pounds=3,4,...,\aleph-1,\aleph,\aleph+2,\aleph+3,...,n$, we see that $\Im_{M}(q_{\pounds}|M_{G}^{\ast}\cup\{r_{\pounds}\})\neq\Im_{M}(r_{\pounds}q_{\pounds}|M_{G}^{\ast}\cup\{r_{\pounds}\})$. From this, we obtain $\Im_{M}(q_{\pounds}|M_{G}^{m})\neq\Im_{M}(r_{\pounds}q_{\pounds}|M_{G}^{m})$, for any $1\leq \pounds \leq n$ and so $|M_{G}^{m}|\leq n+1$, suggesting that $mdim(\mathbb{W}_{n})=n+1$ in this case.\\

\textbf{Case(\rom{2})} $n\equiv1(mod2)$.\\
In this case, $n$ can be written as $n = 2\aleph+1$, where $\aleph \in \mathbb{N}$ and $\aleph \geq 3$. Let $M_{G}^{\ast} = \{p_{1}, r_{1},r_{2},r_{\aleph+1}, \\ r_{\aleph+2}\} \subset \mathbb{V}(\mathbb{W}_{n})$. Now, to figure out that $M_{G}^{\ast}$ is the MMG for the Web graph $\mathbb{W}_{n}$, we consign the mixed metric codes for each vertex and each edge of the graph $\mathbb{W}_{n}$ regarding $M_{G}^{\ast}$. \\

Now, the mixed metric codes for the vertices $\{\upsilon=p_{\pounds}, q_{\pounds}, r_{\pounds}| \pounds=1,2,3,...,n\}$ regarding the set $M_{G}^{\ast}$ are shown below in the following three tables respectively.
\vspace{-0.4mm}
\begin{center}
 \begin{tabular}{{|m{16.0em}|m{22.0em}|}}
 \hline
  $\Im_{M}(\upsilon |M_{G}^{\ast})$ & $M_{G}^{\ast} = \{p_{1}, r_{1},r_{2}, r_{\aleph+1}, r_{\aleph+2}\}$ \\
 \hline
 \vspace{-0.4mm}
 $\Im_{M}(p_{\pounds}|M_{G}^{\ast})$: ($\pounds=1$)                       & $(\pounds-1,\pounds+1,3,\aleph-\pounds+3,\aleph+2)$ \\
 \hline
 $\Im_{M}(p_{\pounds}|M_{G}^{\ast})$: ($2\leq \pounds \leq \aleph+1 $)& $(\pounds-1,\pounds+1,\pounds,\aleph-\pounds+3,\aleph-\pounds+4)$ \\
 \hline
 $\Im_{M}(p_{\pounds}|M_{G}^{\ast})$: ($\pounds=\aleph+2$)   & $(2\aleph-\pounds+2,2\aleph-\pounds+4,\pounds,\pounds-\aleph+1,\aleph-\pounds+4)$ \\
 \hline
 $\Im_{M}(p_{\pounds}|M_{G}^{\ast})$: ($\aleph+3 \leq \pounds \leq 2\aleph+1$)   & $(2\aleph-\pounds+2,2\aleph-\pounds+4,2\aleph-\pounds+5,\pounds-\aleph+1,\pounds-\aleph)$ \\
 \hline
 \end{tabular}
 \end{center}
\vspace{-0.4mm}
\begin{center}
 \begin{tabular}{{|m{16.0em}|m{22.0em}|}}
 \hline
  $\Im_{M}(\upsilon |M_{G}^{\ast})$ & $M_{G}^{\ast} = \{p_{1}, r_{1},r_{2}, r_{\aleph+1}, r_{\aleph+2}\}$ \\
 \hline
 \vspace{-0.4mm}
 $\Im_{M}(q_{\pounds}|M_{G}^{\ast})$: ($\pounds=1$)                       & $(\pounds,\pounds,2,\aleph-\pounds+2,\aleph+1)$ \\
 \hline
 $\Im_{M}(q_{\pounds}|M_{G}^{\ast})$: ($2\leq \pounds \leq \aleph+1$) & $(\pounds,\pounds,\pounds-1,\aleph-\pounds+2,\aleph-\pounds+3)$ \\
 \hline
 $\Im_{M}(q_{\pounds}|M_{G}^{\ast})$: ($\pounds=\aleph+2$)             & $(2\aleph-\pounds+3,2\aleph-\pounds+3,\pounds-1,\pounds-\aleph,\aleph-\pounds+3)$ \\
 \hline
 $\Im_{M}(q_{\pounds}|M_{G}^{\ast})$: ($\aleph+3 \leq \pounds \leq 2\aleph+1$)   & $(2\aleph-\pounds+3,2\aleph-\pounds+3,2\aleph-\pounds+4,\pounds-\aleph,\pounds-\aleph-1)$ \\
 \hline
 \end{tabular}
 \end{center}
\vspace{-0.4mm}
\begin{center}
 \begin{tabular}{{|m{16.0em}|m{22.0em}|}}
 \hline
  $\Im_{M}(\upsilon |M_{G}^{\ast})$ & $M_{G}^{\ast} = \{p_{1}, r_{1},r_{2}, r_{\aleph+1}, r_{\aleph+2}\}$ \\
 \hline
 \vspace{-0.4mm}
 $\Im_{M}(r_{\pounds}|M_{G}^{\ast})$: ($\pounds=1$)                    & $(\pounds+1,0,3,\aleph-\pounds+3,\aleph+2)$ \\
 \hline
 $\Im_{M}(r_{\pounds}|M_{G}^{\ast})$: ($\pounds=2$)                    & $(\pounds+1,\pounds+1,0,\aleph-\pounds+3,\aleph-\pounds+4)$ \\
 \hline
 $\Im_{M}(r_{\pounds}|M_{G}^{\ast})$: ($3\leq \pounds \leq \aleph$)    & $(\pounds+1,\pounds+1,\pounds,\aleph-\pounds+3,\aleph-\pounds+4)$ \\
 \hline
 $\Im_{M}(r_{\pounds}|M_{G}^{\ast})$: ($\pounds=\aleph+1$)             & $(\pounds+1,\pounds+1,\pounds,0,\aleph-\pounds+4)$ \\
 \hline
 $\Im_{M}(r_{\pounds}|M_{G}^{\ast})$: ($\pounds=\aleph+2$)             & $(2\aleph-\pounds+4,2\aleph-\pounds+4,\pounds,\pounds-\aleph+1,0)$ \\
 \hline
 $\Im_{M}(r_{\pounds}|M_{G}^{\ast})$: ($\aleph+3 \leq \pounds \leq 2\aleph+1$)   & $(2\aleph-\pounds+4,2\aleph-\pounds+4,2\aleph-\pounds+5,\pounds-\aleph+1,\pounds-\aleph)$ \\
 \hline
 \end{tabular}
 \end{center}
\vspace{-0.4mm}
and the mixed metric codes for the edges $\{\epsilon=p_{\pounds}p_{\pounds+1}, p_{\pounds}q_{\pounds}, q_{\pounds}q_{\pounds+1}, q_{\pounds}r_{\pounds}| \pounds=1,2,3,...,n\}$ regarding the set $M_{G}^{\ast}$ are shown in the tables below, respectively.
\vspace{-0.4mm}
\begin{center}
 \begin{tabular}{{|m{16.5em}|m{21.5em}|}}
 \hline
  $\Im_{M}(\epsilon |M_{G}^{\ast})$ & $M_{G}^{\ast} = \{p_{1}, r_{1},r_{2}, r_{\aleph+1}, r_{\aleph+2}\}$ \\
 \hline
 \vspace{-0.4mm}
 $\Im_{M}(p_{\pounds}p_{\pounds+1}|M_{G}^{\ast})$: ($ \pounds=1$)                    & $(\pounds-1,\pounds+1,2,\aleph-\pounds+2,\aleph-\pounds+3)$ \\
 \hline
 $\Im_{M}(p_{\pounds}p_{\pounds+1}|M_{G}^{\ast})$: ($2 \leq \pounds \leq \aleph$) & $(\pounds-1,\pounds+1,\pounds,\aleph-\pounds+2,\aleph-\pounds+3)$ \\
 \hline
 $\Im_{M}(p_{\pounds}p_{\pounds+1}|M_{G}^{\ast})$: ($\pounds=\aleph+1$)        & $(\pounds-1,\pounds+1,\pounds,\pounds-\aleph+1,\aleph-\pounds+3)$ \\
 \hline
 $\Im_{M}(p_{\pounds}p_{\pounds+1}|M_{G}^{\ast})$: ($\pounds=\aleph+2$)     & $(2\aleph-\pounds+1,2\aleph-\pounds+3,\pounds,\pounds-\aleph+1,\pounds-\aleph)$ \\
 \hline
 $\Im_{M}(p_{\pounds}p_{\pounds+1}|M_{G}^{\ast})$:($\aleph+3 \leq \pounds \leq 2\aleph+1$)     & $(2\aleph-\pounds+1,2\aleph-\pounds+3,2\aleph-\pounds+4,\pounds-\aleph+1,\pounds-\aleph)$ \\
 \hline
 \end{tabular}
 \end{center}
\vspace{-0.4mm}
\begin{center}
 \begin{tabular}{{|m{16.0em}|m{22.0em}|}}
 \hline
  $\Im_{M}(\epsilon |M_{G}^{\ast})$ & $M_{G}^{\ast} = \{p_{1}, r_{1},r_{2}, r_{\aleph+1}, r_{\aleph+2}\}$ \\
 \hline
 \vspace{-0.4mm}
 $\Im_{M}(p_{\pounds}q_{\pounds}|M_{G}^{\ast})$: ($ \pounds=1$)                    & $(\pounds-1,\pounds,2,\aleph-\pounds+2,\aleph+1)$ \\
 \hline
 $\Im_{M}(p_{\pounds}q_{\pounds}|M_{G}^{\ast})$: ($2\leq \pounds \leq \aleph+1$) & $(\pounds-1,\pounds,\pounds-1,\aleph-\pounds+2,\aleph-\pounds+3)$ \\
 \hline
 $\Im_{M}(p_{\pounds}q_{\pounds}|M_{G}^{\ast})$: ($\pounds=\aleph+2$)     & $(2\aleph-\pounds+2,2\aleph-\pounds+3,\pounds-1,\pounds-\aleph,\aleph-\pounds+3)$ \\
 \hline
 $\Im_{M}(p_{\pounds}q_{\pounds}|M_{G}^{\ast})$: ($\aleph+3 \leq \pounds \leq 2\aleph+1$)     & $(2\aleph-\pounds+2,2\aleph-\pounds+3,2\aleph-\pounds+4,\pounds-\aleph,\pounds-\aleph-1)$ \\
 \hline
 \end{tabular}
 \end{center}
\vspace{-0.4mm}
\begin{center}
 \begin{tabular}{{|m{16.5em}|m{21.5em}|}}
 \hline
  $\Im_{M}(\epsilon |M_{G}^{\ast})$ & $M_{G}^{\ast} = \{p_{1}, r_{1},r_{2}, r_{\aleph+1}, r_{\aleph+2}\}$ \\
 \hline
 \vspace{-0.4mm}
 $\Im_{M}(q_{\pounds}q_{\pounds+1}|M_{G}^{\ast})$: ($ \pounds=1$)                    & $(\pounds,\pounds,1,\aleph-\pounds+1,\aleph-\pounds+2)$ \\
 \hline
 $\Im_{M}(q_{\pounds}q_{\pounds+1}|M_{G}^{\ast})$: ($2\leq \pounds \leq \aleph$) & $(\pounds,\pounds,\pounds-1,\aleph-\pounds+1,\aleph-\pounds+2)$ \\
 \hline
 $\Im_{M}(q_{\pounds}q_{\pounds+1}|M_{G}^{\ast})$: ($\pounds=\aleph+1$)     & $(\pounds,\pounds,\pounds-1,\pounds-\aleph,\aleph-\pounds+2)$ \\
 \hline
 $\Im_{M}(q_{\pounds}q_{\pounds+1}|M_{G}^{\ast})$: ($\pounds=\aleph+2$)  & $(2\aleph-\pounds+2,2\aleph-\pounds+2,\pounds-1,\pounds-\aleph,\pounds-\aleph-1)$ \\
 \hline
 $\Im_{M}(q_{\pounds}q_{\pounds+1}|M_{G}^{\ast})$:($\aleph+3 \leq \pounds \leq 2\aleph+1$)     & $(2\aleph-\pounds+2,2\aleph-\pounds+2,2\aleph-\pounds+3,\pounds-\aleph,\pounds-\aleph-1)$ \\
 \hline
 \end{tabular}
 \end{center}
\vspace{-0.4mm}
\begin{center}
 \begin{tabular}{{|m{16.0em}|m{22.0em}|}}
 \hline
  $\Im_{M}(\epsilon |M_{G}^{\ast})$ & $M_{G}^{\ast} = \{p_{1}, r_{1},r_{2}, r_{\aleph+1}, r_{\aleph+2}\}$ \\
 \hline
 \vspace{-0.4mm}
 $\Im_{M}(q_{\pounds}r_{\pounds}|M_{G}^{\ast})$: ($ \pounds=1$)                    & $(\pounds,0,2,\aleph-\pounds+2,\aleph+1)$ \\
 \hline
 $\Im_{M}(q_{\pounds}r_{\pounds}|M_{G}^{\ast})$: ($\pounds=2$)                     & $(\pounds,\pounds,0,\aleph-\pounds+2,\aleph-\pounds+3)$ \\
 \hline
 $\Im_{M}(q_{\pounds}r_{\pounds}|M_{G}^{\ast})$: ($3 \leq \pounds \leq \aleph$) & $(\pounds,\pounds,\pounds-1,\aleph-\pounds+2,\aleph-\pounds+3)$ \\
 \hline
 $\Im_{M}(q_{\pounds}r_{\pounds}|M_{G}^{\ast})$: ($\pounds=\aleph+1$)     & $(\pounds,\pounds,\pounds-1,0,\aleph-\pounds+3)$ \\
 \hline
 $\Im_{M}(q_{\pounds}r_{\pounds}|M_{G}^{\ast})$: ($\pounds=\aleph+2$)     & $(2\aleph-\pounds+3,2\aleph-\pounds+3,\pounds-1,\pounds-\aleph,0)$ \\
 \hline
 $\Im_{M}(q_{\pounds}r_{\pounds}|M_{G}^{\ast})$: ($\aleph+3 \leq \pounds \leq 2\aleph+1$)     & $(2\aleph-\pounds+3,2\aleph-\pounds+3,2\aleph-\pounds+4,\pounds-\aleph,\pounds-\aleph-1)$ \\
 \hline
 \end{tabular}
 \end{center}
 \vspace{-0.4mm}

Now, from these mixed metric codes of the edges and the vertices of the Web graph $\mathbb{W}_{n}$ concerning the set $M_{G}^{\ast}$, we ascertain that for $1\leq \pounds \leq n$ and $\pounds\neq 1,2, \aleph+1, \aleph+2$, $\Im_{M}(q_{\pounds}|M_{G}^{\ast})=\Im_{M}(r_{\pounds}q_{\pounds}|M_{G}^{\ast})$. For the remaining mixed metric edges and vertices codes in $\mathbb{W}_{n}$, we find no two vertices or edges with the same mixed metric codes. For $\pounds=3,4,...,\aleph-1,\aleph,\aleph+2,\aleph+3,...,n$, we see that $\Im_{M}(q_{\pounds}|M_{G}^{\ast}\cup\{r_{\pounds}\})\neq\Im_{M}(r_{\pounds}q_{\pounds}|M_{G}^{\ast}\cup\{r_{\pounds}\})$. From this, we obtain $\Im_{M}(q_{\pounds}|M_{G}^{m})\neq\Im_{M}(r_{\pounds}q_{\pounds}|M_{G}^{m})$, for any $1\leq \pounds \leq n$ and so $|M_{G}^{m}|\leq n+1$, suggesting that $mdim(\mathbb{W}_{n})=n+1$ in this case also, which concludes the theorem.
\end{proof}

\hspace{-5.0mm}\textbf{Theorem 8.}
{\it The independent mixed metric number for the Web graph $\mathbb{W}_{n}$, for $n\geq4$ is $n+1$.}
\begin{proof}
  For proof, refer to Theorem $7$.
\end{proof}

\section{Conclusion}

In this examination, we determined the MMD for the two important families of the plane graphs viz., the Web graph $\mathbb{W}_{n}$ (\cite{ys}, see Figure 2) and the Prism allied graph $\mathbb{D}_{n}^{t}$ (\cite{fwr}, see Figure 1), and which was found to be non-constant unbounded for these two families of the plane graph. Moreover, for the Web graph $\mathbb{W}_{n}$ and the Prism allied graph $\mathbb{D}_{n}^{t}$, we unveil that the mixed metric basis set $M_{G}^{m}$ is independent. From preliminaries and these results, for these two families of plane graphs $H=D^{t}_{n}$ and $H=\mathbb{W}_{n}$, we determined that $\beta(H)<\beta_{E}(H)<\beta_{M}(H)$, for every $n\geq5$.

\end{document}